\documentclass[11pt]{article}
\usepackage[margin=1in]{geometry}

\usepackage{amsmath} % assumes amsmath package installed
\usepackage{amssymb}  % assumes amsmath package installed
\usepackage{amsthm}

\usepackage{scalefnt}
\usepackage{lscape}  %AAA: adding this for a large matrix in the appendix.
\usepackage{color}

\makeatletter
\renewcommand*\env@matrix[1][c]{\hskip -\arraycolsep
  \let\@ifnextchar\new@ifnextchar
  \array{*\c@MaxMatrixCols #1}}
\makeatother

\theoremstyle{plain}  % Bold name, italics font
\newtheorem{theorem}{Theorem}[section]
\newtheorem{lemma}[theorem]{Lemma}

\newtheorem{corollary}[theorem]{Corollary}
\newtheorem{definition}[theorem]{Definition}

\theoremstyle{definition}

\theoremstyle{remark} % italics name, roman font

\newtheorem{remark}{Remark}[section]

\long\def\aaa#1{{#1}}
\long\def\aaan#1{{#1}}

\title{\LARGE \bf A Complete Characterization of the Gap between \\ Convexity and SOS-Convexity}
 \author{Amir Ali Ahmadi and Pablo A. Parrilo \thanks{The authors are with the Laboratory for Information and
Decision Systems, Department of Electrical Engineering and
Computer Science, Massachusetts Institute of Technology. Email:
\{\texttt{a\_a\_a}, \texttt{parrilo}\}\texttt{@mit.edu}.  }
\thanks{This research was partially supported by the NSF Focused
Research Group Grant on Semidefinite Optimization and Convex
Algebraic Geometry DMS-0757207.} }
\begin{document}
\date{}
\maketitle

%\thispagestyle{empty}
%\pagestyle{empty}

%%%%%%%%%%%%%%%%%%%%%%%%%%%%%%%%%%%%%%%%%%%%%%%%%%%%%%%%%%%%%%%%%%%%%%%%%%%%%%%%
\begin{abstract}
Our first contribution in this paper is to prove that three
natural sum of squares (sos) based sufficient conditions for
convexity of polynomials\aaan{,} via the definition of convexity,
its first order characterization, and its second order
characterization\aaan{,} are equivalent. These three equivalent
algebraic conditions, henceforth referred to as sos-convexity, can
be checked by semidefinite programming whereas deciding convexity
is NP-hard. If we denote the set of convex and sos-convex
polynomials in $n$ variables of degree $d$ with $\tilde{C}_{n,d}$
and $\tilde{\Sigma C}_{n,d}$ respectively, then our main
contribution is to prove that $\tilde{C}_{n,d}=\tilde{\Sigma
C}_{n,d}$ if and only if $n=1$ or $d=2$ or $(n,d)=(2,4)$. We also
present a complete characterization for forms (homogeneous
polynomials) except for the case $(n,d)=(3,4)$ which is joint work
with G. Blekherman and is to be published elsewhere. Our result
states that the set $C_{n,d}$ of convex forms in $n$ variables of
degree $d$ equals the set $\Sigma C_{n,d}$ of sos-convex forms if
and only if $n=2$ or $d=2$ or $(n,d)=(3,4)$. To prove these
results, we present in particular explicit examples of polynomials
in $\tilde{C}_{2,6}\setminus\tilde{\Sigma C}_{2,6}$ and
$\tilde{C}_{3,4}\setminus\tilde{\Sigma C}_{3,4}$ and forms in
$C_{3,6}\setminus\Sigma C_{3,6}$ and $C_{4,4}\setminus\Sigma
C_{4,4}$, and a general procedure for constructing forms in
$C_{n,d+2}\setminus\Sigma C_{n,d+2}$ from nonnegative but not sos
forms in $n$ variables and degree $d$.

\noindent Although for disparate reasons, the remarkable outcome
is that convex polynomials (resp. forms) are sos-convex exactly in
cases where nonnegative polynomials (resp. forms) are sums of
squares, as characterized by Hilbert.

\end{abstract}

%%%%%%%%%%%%%%%%%%%%%%%%%%%%%%%%%%%%%%%%%%%%%%%%%%%%%%%%%%%%%%%%%%%%%%%%%%%%%%%%
\section{Introduction}
\subsection{Nonnegativity and sum of squares}
One of the cornerstones of real algebraic geometry is Hilbert's
seminal paper in 1888~\cite{Hilbert_1888}, where he gives a
complete characterization of the degrees and dimensions \aaan{for}
which nonnegative polynomials can be written as sums of squares of
polynomials. In particular, Hilbert proves in~\cite{Hilbert_1888}
that there exist nonnegative polynomials that are not sums of
squares, although explicit examples of such polynomials appeared
only about 80 years later and the study of the gap between
nonnegative and sums of squares polynomials continues to be an
active area of research to this day.

%
%The main results of Hilbert's paper include the rather surprising
%fact that all nonnegative bivariate quartic polynomials are sums
%of squares and the first proof of existence of nonnegative
%polynomials that are not sums of squares, which he specifically
%showed must exist within bivariate sextic and ternary quartic
%polynomials. Interestingly, explicit examples of such polynomials
%appeared only about 80 years later with the first examples due
%respectively to Motzkin~\cite{MotzkinSOS} and
%Robinson~\cite{RobinsonSOS}; see~\cite{Reznick}.

Motivated by a wealth of new applications and a modern viewpoint
that emphasizes efficient computation, there has also been a great
deal of recent interest from the optimization community in the
representation of nonnegative polynomials as sums of squares
(sos). Indeed, many fundamental problems in applied and
computational mathematics can be reformulated as either deciding
whether certain polynomials are nonnegative or searching over a
family of nonnegative polynomials. It is well-known however that
if the degree of the polynomial is four or larger, deciding
nonnegativity is an NP-hard problem (\aaan{this follows, e.g.,} as
an immediate corollary of NP-hardness of deciding matrix
copositivity~\cite{nonnegativity_NP_hard}). On the other hand, it
is also well-known that deciding whether a polynomial can be
written as a sum of squares can be reduced to solving a
semidefinite program, for which efficient algorithms, e.g. based
on interior point methods, are available. The general machinery of
the so-called ``sos relaxation'' has therefore been to replace the
intractable nonnegativity requirements with the more tractable sum
of squares requirements that obviously provide a sufficient
condition for polynomial nonnegativity.

Some relatively recent applications that sum of squares
relaxations have found span areas as diverse as control
theory~\cite{PhD:Parrilo},~\cite{PositivePolyInControlBook},
quantum computation~\cite{Pablo_Sep_Entang_States}, polynomial
games~\cite{Pablo_poly_games}, combinatorial
optimization~\cite{Stability_number_SOS}, and many others.

\subsection{Convexity and sos-convexity}

Aside from nonnegativity, \emph{convexity} is another fundamental
property of polynomials that is of both theoretical and practical
significance. Perhaps most notably, presence of convexity in an
optimization problem often leads to tractability of finding global
optimal solutions. Consider for example the problem of finding the
unconstrained global minimum of a polynomial. This is an NP-hard
problem in general~\cite{Minimize_poly_Pablo}, but if we know a
priori that the polynomial to be minimized is convex, then every
local \aaa{minimum} is global, and even simple gradient descent
methods can find a global minimum. There are other scenarios where
one would like to decide convexity of polynomials. For example, it
turns \aaan{out} that the $d$-th root of a degree $d$ polynomial
is a norm if and only if the polynomial is homogeneous, positive
definite, and convex~\cite{Blenders_Reznick}. Therefore, if we can
certify that a homogenous polynomial is convex and definite, then
we can use it to define a norm, which is useful in many
applications. In many other practical settings, we might want to
\emph{parameterize} a family of convex polynomials that have
certain properties, e.g., that serve as a convex envelope for a
non-convex function, approximate a more complicated function, or
fit some data points with minimum error. In the field of robust
control for example, it is common to \aaan{use} convex Lyapunov
functions to prove stability of uncertain dynamical systems
described \aaan{by} difference inclusions. Therefore, the ability
to efficiently search over convex polynomials would lead to
algorithmic ways of constructing such Lyapunov functions.

%
%
%to deal with uncertain dynamical systems that are described as a
%system of difference inclusions. Stability of such systems over
%the entire uncertainty set can be proven by construction of a
%convex Lyapunov function.

The question of determining the computational complexity of
deciding convexity of polynomials appeared in 1992 on a list of
seven open problems in complexity theory for numerical
optimization~\cite{open_complexity}. In a recent joint work with
A. Olshevsky and J. N. Tsitsiklis, we have shown that the problem
is strongly NP-hard even for polynomials of degree
four~\cite{NPhard_Convexity_arxiv}. If testing membership in the
set of convex polynomials is hard, searching or optimizing over
them is obviously also hard. This result, like any other hardness
result, stresses the need for good approximation algorithms that
can deal with many instances of the problem efficiently.

The focus of this work is on an algebraic notion known as
\emph{sos-convexity} (introduced formally by Helton and Nie
in~\cite{Helton_Nie_SDP_repres_2}), which is a sufficient
condition for convexity of polynomials based on \aaan{a} sum of
squares decomposition of the Hessian matrix; see
Definition~\ref{def:sos.convex}. As we will briefly review in
Section~\ref{sec:prelims}, the problem of deciding if a given
polynomial is sos-convex amounts to solving a single semidefinite
program.

Besides its computational implications, sos-convexity is an
appealing concept since it bridges the geometric and algebraic
aspects of convexity. Indeed, while the usual definition of
convexity is concerned only with the geometry of the epigraph, in
sos-convexity this geometric property (or the nonnegativity of the
Hessian) must be certified through a ``simple'' algebraic
identity, namely the sum of squares factorization of the Hessian.
The original motivation of Helton and Nie for defining
sos-convexity was in relation to the question of semidefinite
representability of convex sets~\cite{Helton_Nie_SDP_repres_2}.
But this notion has already appeared in the literature in a number
of other
settings~\cite{Lasserre_Jensen_inequality},~\cite{Lasserre_Convex_Positive},~\cite{convex_fitting},~\cite{Chesi_Hung_journal}.
In particular, there has been much recent interest in the role of
convexity in semialgebraic geometry
~\cite{Lasserre_Jensen_inequality},~\cite{Blekherman_convex_not_sos},~\cite{Monique_Etienne_Convex},~\cite{Lasserre_set_convexity}
and sos-convexity is a recurrent figure in this line of research.

\subsection{Contributions and organization of the paper}
The main contribution of this work is to establish the counterpart
of Hilbert's characterization of the gap between nonnegativity and
sum of squares for the notions of convexity and sos-convexity. We
start by presenting some background material in
Section~\ref{sec:prelims}. In
Section~\ref{sec:equiv.defs.of.sos.convexity}, we prove an
algebraic analogue of a classical result in convex analysis, which
provides three equivalent characterizations for sos-convexity
(Theorem~\ref{thm:sos.convexity.3.equivalent.defs}). This result
substantiates the fact that sos-convexity is \emph{the} right sos
relaxation for convexity. In Section~\ref{sec:first.examples}, we
present some examples of convex polynomials that are not
sos-convex. In Section~\ref{sec:full.characterization}, we provide
the characterization of the gap between convexity and
sos-convexity (Theorem~\ref{thm:full.charac.polys} and
Theorem~\ref{thm:full.charac.forms}).
Subsection~\ref{subsec:proof.equal.cases} includes the proofs of
the cases where convexity and sos-convexity are equivalent and
Subsection~\ref{subsec:proof.non.equal.cases} includes the proofs
of the cases where they are not. In particular,
Theorem~\ref{thm:minimal.2.6.and.3.6} and
Theorem~\ref{thm:minimal.3.4.and.4.4} present explicit examples of
convex but not sos-convex polynomials that have dimension and
degree as low as possible, and
Theorem~\ref{thm:conv_not_sos_conv_forms_n3d} provides a general
construction for producing such polynomials in higher degrees.
Some concluding remarks and an open problem are presented in
Section~\ref{sec:concluding.remarks}.

\section{Preliminaries}\label{sec:prelims}
\subsection{Background on nonnegativity and sum of squares}\label{subsec:nonnegativity.sos.basics}

A (multivariate) \emph{polynomial} $p\mathrel{\mathop:}=p(x)$ in
variables $x\mathrel{\mathop:}=(x_1,\ldots,x_n)^T$ is a function
from $\mathbb{R}^n$ to $\mathbb{R}$ that is a finite linear
combination of monomials:
\begin{equation}\nonumber
p(x)=\sum_{\alpha}c_\alpha x^\alpha=\sum_{\aaan{(}\alpha_1,
\ldots, \alpha_n\aaan{)}} c_{\alpha_1,\ldots,\alpha_n}
x_1^{\alpha_1} \cdots x_n^{\alpha_n} ,
\end{equation}
where the sum is over $n$-tuples of nonnegative integers
\aaa{$\alpha\mathrel{\mathop:}=(\alpha_1,\ldots,\alpha_n)$}. We
will be concerned throughout with polynomials with real
coefficients, i.e., we will have $c_\alpha\in\mathbb{R}$. The ring
of polynomials in $n$ variables with real coefficients is denoted
by $\mathbb{R}[x]$. The \emph{degree} of a monomial $x^\alpha$ is
equal to $\alpha_1 + \cdots + \alpha_n$. The degree of a
polynomial $p\in\mathbb{R}[x]$ is defined to be the highest degree
of its component monomials. A polynomial $p$ is said to be
\emph{nonnegative} or \emph{positive semidefinite (psd)} if
$p(x)\geq0$ for all $x\in\mathbb{R}^n$. Clearly, a necessary
condition for a polynomial to be psd is for its degree to be even.
We say that $p$ is a \emph{sum of squares (sos)}, if there exist
polynomials $q_{1},\ldots,q_{m}$ such that
$p=\sum_{i=1}^{m}q_{i}^{2}$. We denote the set of psd (resp. sos)
polynomials in $n$ variables and degree $d$ by $\tilde{P}_{n,d}$
(resp. $\tilde{\Sigma}_{n,d}$). Any sos polynomial is clearly psd,
so we have $\tilde{\Sigma}_{n,d}\subseteq \tilde{P}_{n,d}$.
%\begin{equation}\label{eq:sos.decomp.q_i}
%p(x)=\sum_{i=1}^{m}q_{i}^{2}(x).
%\end{equation}

A \emph{homogeneous polynomial} (or a \emph{form}) is a polynomial
where all the monomials have the same degree. A form $p$ of degree
$d$ is a homogeneous function of degree $d$ since it satisfies
$p(\lambda x)=\lambda^d p(x)$ for any scalar
$\lambda\in\mathbb{R}$. We say that a form $p$ is \emph{positive
definite} if $p(x)>0$ for all $x\neq0$ in $\mathbb{R}^n$.
Following standard notation, we denote the set of psd (resp. sos)
homogeneous polynomials in $n$ variables and degree $d$ by
$P_{n,d}$ (resp. $\Sigma_{n,d}$). Once again, we have the obvious
inclusion $\Sigma_{n,d}\subseteq P_{n,d}$. All of the four sets
$\Sigma_{n,d}, P_{n,d}, \tilde{\Sigma}_{n,d}, \tilde{P}_{n,d}$ are
closed convex cones. The closedness of the sum of squares cone may
not be so obvious. This fact was first proved by
Robinson~\cite{RobinsonSOS}. We will make crucial use of it in the
proof of Theorem~\ref{thm:sos.convexity.3.equivalent.defs} in the
next section.

Any form of degree $d$ in $n$ variables can be ``dehomogenized''
into a polynomial of degree $\leq d$ in $n-1$ variables by setting
$x_n=1$. Conversely, any polynomial $p$ of degree $d$ in $n$
variables can be ``homogenized'' into a form $p_h$ of degree $d$
in $n+1$ variables, by adding a new variable $y$, and letting $
p_h(x_1,\ldots,x_n,y)\mathrel{\mathop:}=y^{d} \, p\left({x_1}/{y},
\ldots, {x_n}/{y}\right)$.  The properties of being psd and sos
are preserved under homogenization and
dehomogenization~\cite{Reznick}.

A very natural and fundamental question that as we mentioned
earlier was answered by Hilbert is to understand \aaan{for} what
dimensions and degrees nonnegative polynomials (or forms) can
\aaan{always} be represented as sums of squares, i.e, for what
values of $n$ and $d$ we have
$\tilde{\Sigma}_{n,d}=\tilde{P}_{n,d}$ or $\Sigma_{n,d}=P_{n,d}$.
Note that because of the argument in the last paragraph, we have
$\tilde{\Sigma}_{n,d}=\tilde{P}_{n,d}$ if and only if
$\Sigma_{n+1,d}=P_{n+1,d}$. Hence, it is enough to answer the
question just for polynomials or just for forms\aaa{.}
%and the answer to the other one comes for free.

\begin{theorem}[Hilbert,~\cite{Hilbert_1888}]\label{thm:Hilbert}
$\tilde{\Sigma}_{n,d}=\tilde{P}_{n,d}$ if and only if $n=1$ or
$d=2$ or $(n,d)=(2,4)$. Equivalently, $\Sigma_{n,d}=P_{n,d}$ if
and only if $n=2$ or $d=2$ or $(n,d)=(3,4)$.
\end{theorem}

The proofs of $\tilde{\Sigma}_{1,d}=\tilde{P}_{1,d}$ and
$\tilde{\Sigma}_{n,2}=\tilde{P}_{n,2}$ are relatively simple and
were known before Hilbert. On the other hand, the proof of the
fairly surprising fact that $\tilde{\Sigma}_{2,4}=\tilde{P}_{2,4}$
(or equivalently $\Sigma_{3,4}=P_{3,4}$) is more involved. We
refer the interested reader to
\cite{NewApproach_Hilbert_Ternary_Quatrics},
\cite{Scheiderer_ternary_quartic},
\cite{Choi_Lam_extremalPSDforms}, and references in~\cite{Reznick}
for some modern expositions and alternative proofs of this result.
Hilbert's other main contribution was to show that these are the
only cases where nonnegativity and sum of squares are equivalent
by giving a nonconstructive proof of existence of polynomials in
$\tilde{P}_{2,6}\setminus\tilde{\Sigma}_{2,6}$ and
$\tilde{P}_{3,4}\setminus\tilde{\Sigma}_{3,4}$ (or equivalently
forms in $P_{3,6}\setminus\Sigma_{3,6}$ and
$P_{4,4}\setminus\Sigma_{4,4}$). From this, it follows with simple
arguments that in all higher dimensions and degrees there must
also be psd but not sos polynomials; see~\cite{Reznick}. Explicit
examples of such polynomials appeared in the 1960s starting from
the celebrated Motzkin form~\cite{MotzkinSOS}:
\begin{equation}\label{eq:Motzkin.form}
M(x_1,x_2,x_3)=x_1^4x_2^2+x_1^2x_2^4-3x_1^2x_2^2x_3^2+x_3^6,
\end{equation}
which belongs to $P_{3,6}\setminus\Sigma_{3,6}$, and continuing a
few years later with the Robinson form~\cite{RobinsonSOS}:
\begin{equation}\label{eq:Robinston.form}
R(x_1,x_2,x_3,x_4)=x_1^2(x_1-x_4)^2+x_2^2(x_2-x_4)^2+x_3^2(x_3-x_4)^2+2x_1x_2x_3(x_1+x_2+x_3-2x_4),
\end{equation}
which belongs to $P_{4,4}\setminus\Sigma_{4,4}$.

Several other constructions of psd polynomials that are not sos
have appeared in the literature since. An excellent survey
is~\cite{Reznick}. See also~\cite{Reznick_Hilbert_construciton}
and~\cite{Blekherman_nonnegative_and_sos}.

%A polynomial $p(x)$ of degree $d$ in $n$ variables has
%$n+d\choose{d}$ coefficients, whereas the number of coefficients
%of a form of degree $d$ in $n$ variables is $n+d-1\choose{d}$.
%It is easy to show that if a form of degree $d$ is sos, then $d$
%is even and the polynomials $q_i$ in the sos decomposition are
%forms of degree $d/2$.

\subsection{Connection to semidefinite programming and matrix
generalizations}\label{subsec:sos.sdp.and.matrix.generalize} As we
remarked before, what makes sum of squares an appealing concept
from a computational viewpoint is its relation to semidefinite
programming. It is well-known (see e.g. \cite{PhD:Parrilo},
\cite{sdprelax}) that a polynomial $p$ in $n$ variables and of
even degree $d$ is a sum of squares if and only if there exists a
positive semidefinite matrix $Q$ (often called the Gram matrix)
such that
$$p(x)=z^{T}Qz,$$
where $z$ is the vector of monomials of degree up to $d/2$
\begin{equation}\label{eq:monomials}
z=[1,x_{1},x_{2},\ldots,x_{n},x_{1}x_{2},\ldots,x_{n}^{d/2}].
\end{equation}
The set of all such matrices $Q$ is the feasible set of a
semidefinite program (SDP). For fixed $d$, the size of this
semidefinite program is polynomial in $n$. Semidefinite programs
can be solved with arbitrary accuracy in polynomial time. There
are several implementations of semidefinite programming solvers,
based on interior point algorithms among others, that are very
efficient in practice and widely used; see~\cite{VaB:96} and
references therein.

%The conversion step of going from an sos decomposition problem to
%an SDP problem is fully algorithmic and has been implemented in
%software packages such as SOSTOOLS~\cite{sostools} and
%YALMIP~\cite{yalmip}.

The notions of positive semidefiniteness and sum of squares of
scalar polynomials can be naturally extended to polynomial
matrices, i.e., matrices with entries in $\mathbb{R}[x]$. We say
that a symmetric polynomial matrix $U(x)\in \mathbb{R}[x]^{m
\times m}$ is positive semidefinite if $U(x)$ is positive
semidefinite in the matrix sense for all $x\in \mathbb{R}^n$, i.e,
if $U(x)$ has nonnegative eigenvalues for all $x\in \mathbb{R}^n$.
It is straightforward to see that this condition holds if and only
if the scalar polynomial $y^{T}U(x)y$ in $m+n$ variables $[x; y]$
is psd. A homogeneous polynomial matrix $U(x)$ is said to be
positive definite, if it is positive definite in the matrix sense,
i.e., has positive eigenvalues, for all $x\neq0$ in
$\mathbb{R}^n$. The definition of an sos-matrix is as follows
\cite{Kojima_SOS_matrix}, \cite{Symmetry_groups_Gatermann_Pablo},
\cite{matrix_sos_Hol}.

\begin{definition}\label{def:sos-matrix}
A symmetric polynomial matrix $U(x)\in~\mathbb{R}[x]^{m \times
m}$,$ \ x\in \mathbb{R}^n,$ is an \emph{sos-matrix} if there
exists a polynomial matrix $V(x)\in \mathbb{R}[x]^{s \times m}$
for some $s\in\mathbb{N}$, such that $P(x)~=~V^{T}(x)V(x)$.
\end{definition}

It turns out that a polynomial matrix $U(x)\in \mathbb{R}[x]^{m
\times m}$, $\ x\in~\mathbb{R}^n,$ is an sos-matrix if and only if
the scalar polynomial $y^{T}U(x)y$ is a sum of squares in
$\mathbb{R}[x; y]$; see~\cite{Kojima_SOS_matrix}. This is a useful
fact because in particular it gives us an easy way of checking
whether a polynomial matrix is an sos-matrix by solving a
semidefinite program. Once again, it is obvious that being an
sos-matrix is a sufficient condition for a polynomial matrix to be
positive semidefinite.

%The converse is not true except for special cases, e.g. for
%univariate matrices~\cite{CLRrealzeros},~\cite{SOS_KYP}.

\subsection{Background on convexity and sos-convexity}\label{subsec:convexity.sos.convexity.basics}

A polynomial $p$ is (globally) convex if for all $x$ and $y$ in
$\mathbb{R}^n$ and all $\lambda \in [0,1]$, we have
\begin{equation}\label{eq:convexity.defn.}
p(\lambda x+(1-\lambda)y)\leq \lambda p(x)+(1-\lambda)p(y).
\end{equation}
Since polynomials are continuous functions, the inequality in
(\ref{eq:convexity.defn.}) holds if and only if it holds for a
fixed value of $\lambda\in(0,1)$, say, $\lambda=\frac{1}{2}$. In
other words, $p$ is convex if and only if
\begin{equation}\label{eq:convexity.with.lambda.0.5}
p\left(\textstyle{\frac{1}{2}}
x+\textstyle{\frac{1}{2}}y\right)\leq \textstyle{\frac{1}{2}
p(x)}+\textstyle{\frac{1}{2}}p(y)
\end{equation} for all $x$ and
$y$; see e.g.~\cite[p. 71]{Rudin_RealComplexAnalysis}. Except for
the trivial case of linear polynomials, an odd degree polynomial
is clearly never convex.

For the sake of direct comparison with a result that we derive in
the next section
(Theorem~\ref{thm:sos.convexity.3.equivalent.defs}), we recall
next a classical result from convex analysis on the first and
second order characterization of convexity. The proof can be found
in many convex optimization textbooks, e.g.~\cite[p.
70]{BoydBook}. The theorem is of course true for any twice
differentiable function, but for our purposes we state it for
polynomials.

\begin{theorem}\label{thm:classical.first.2nd.order.charac.}
Let $p\mathrel{\mathop:}=p(x)$ be a polynomial. Let $\nabla
p\mathrel{\mathop:}=\nabla p(x)$ denote its gradient and let
$H\mathrel{\mathop:}=H(x)$ be its Hessian, i.e., the $n \times n$
symmetric matrix of second derivatives. Then the following are
equivalent.

\textbf{(a)} $p\left(\textstyle{\frac{1}{2}}
x+\textstyle{\frac{1}{2}}y\right)\leq \textstyle{\frac{1}{2}
p(x)}+\textstyle{\frac{1}{2}}p(y),\quad \forall
x,y\in\mathbb{R}^n$; (i.e., $p$ is convex).

\textbf{(b)} $p(y) \geq p(x)+\nabla p(x)^T(y-x),\quad \forall
x,y\in\mathbb{R}^n$\aaan{; (i.e., $p$ lies above the supporting
hyperplane at every point).}

\textbf{(c)} $y^TH(x)y\geq0,\quad \forall x,y\in\mathbb{R}^n$;
(i.e., $H(x)$ is a positive semidefinite polynomial matrix).
\end{theorem}

Helton and Nie proposed in~\cite{Helton_Nie_SDP_repres_2} the
notion of \emph{sos-convexity} as an sos relaxation for the second
order characterization of convexity (condition \textbf{(c)}
above).

\begin{definition}\label{def:sos.convex}
A polynomial $p$ is \emph{sos-convex} if its Hessian
$H\mathrel{\mathop:}=H(x)$ is an sos-matrix.
\end{definition}

With what we have discussed so far, it should be clear that
sos-convexity is a sufficient condition for convexity of
polynomials \aaa{and} can be checked with semidefinite
programming. In the next section, we will show some other natural
sos relaxations for polynomial convexity, which will turn out to
be equivalent to sos-convexity.

We end this section by introducing some final notation:
$\tilde{C}_{n,d}$ and $\tilde{\Sigma C}_{n,d}$ will respectively
denote the set of convex and sos-convex polynomials in $n$
variables and degree $d$; $C_{n,d}$ and $\Sigma C_{n,d}$ will
respectively denote set of convex and sos-convex homogeneous
polynomials in $n$ variables and degree $d$. Again, these four
sets are closed convex cones and we have the obvious inclusions
$\tilde{\Sigma C}_{n,d}\subseteq\tilde{C}_{n,d}$ and $\Sigma
C_{n,d}\subseteq C_{n,d}$.

\section{Equivalent algebraic relaxations for convexity of
polynomials}\label{sec:equiv.defs.of.sos.convexity} An obvious way
to formulate alternative sos relaxations for convexity of
polynomials is to replace every inequality in
Theorem~\ref{thm:classical.first.2nd.order.charac.} with its sos
version. In this section we examine how these relaxations relate
to each other. We also comment on the size of the resulting
semidefinite programs.

Our result below can be thought of as an algebraic analogue of
Theorem~\ref{thm:classical.first.2nd.order.charac.}.
\begin{theorem} \label{thm:sos.convexity.3.equivalent.defs}
Let $p\mathrel{\mathop:}=p(x)$ be a polynomial of degree $d$ in
$n$ variables with its gradient and Hessian denoted respectively
by $\nabla p\mathrel{\mathop:}=\nabla p(x) $ and
$H\mathrel{\mathop:}=H(x)$. Let $g_{\lambda}$, $g_\nabla$, and
$g_{\nabla^2}$ be defined as
\begin{equation} \label{eq:defn.g_lambda.g_grad.g_grad2}
\begin{array}{lll}
g_{\lambda}(x,y)&=&(1-\lambda)p(x)+\lambda p(y)-p((1-\lambda)
x+\lambda y),\\
g_\nabla(x,y)&=&p(y)-p(x)-\nabla p(x)^T(y-x), \\
g_{\nabla^2}(x,y)&=&y^{T}H(x)y.
\end{array}
\end{equation}
Then the following are equivalent:

\textbf{(a)}  \  $g_{\frac{1}{2}}(x,y)$ is sos\footnote{The
constant $\frac{1}{2}$ in $g_{\frac{1}{2}}(x,y)$ of condition
\textbf{(a)} is arbitrary and chosen for convenience. One can show
that $g_{\frac{1}{2}}$ being sos implies that $g_{\lambda}$ is sos
for any fixed $\lambda\in[0,1]$. Conversely, if $g_{\lambda}$ is
sos for some $\lambda\in(0,1)$, then $g_{\frac{1}{2}}$ is sos. The
proofs are similar to the proof of \textbf{(a)$\Rightarrow$(b)}.
}.

%\textbf{(a')} $g_\lambda(x,y)$ is sos for any fixed
%$\lambda\in[0,1].$

\textbf{(b)}  \  $g_\nabla(x,y)$ is sos.

\textbf{(c)}   \  $g_{\nabla^2}(x,y)$ is sos; (i.e., $H(x)$ is an
sos-matrix).
\end{theorem}

\begin{proof}
\textbf{(a)$\Rightarrow$(b)}: Assume $g_{\frac{1}{2}}$ is sos. We
start by proving that $g_{\frac{1}{2^k}}$ will also be sos for any
integer $k\geq2$. A little bit of straightforward algebra yields
the relation
\begin{equation}\label{eq:g_diatic_relation}
g_{\frac{1}{2^{k+1}}}(x,y)=\textstyle{\frac{1}{2}}g_{\frac{1}{2^{k}}}(x,y)+g_{\frac{1}{2}}\left(x,\textstyle{\frac{2^{k}-1}{2^{k}}}x+\textstyle{\frac{1}{2^{k}}}y\right).
\end{equation}
The second term on the right hand side of
(\ref{eq:g_diatic_relation}) is always sos because
$g_{\frac{1}{2}}$ is sos. Hence, this relation shows that for any
$k$, if $g_{\frac{1}{2^k}}$ is sos, then so is
$g_{\frac{1}{2^{k+1}}}$. Since for $k=1$, both terms on the right
hand side of (\ref{eq:g_diatic_relation}) are sos by assumption,
induction immediately gives that $g_{\frac{1}{2^k}}$ is sos for
all $k$.

%When $k=2$, both terms on the right hand side of this equation are
%sos by assumption and hence (\ref{eq:g_diatic_relation}) implies
%that $g_{\frac{1}{4}}$ must be sos. Once this is established, we
%can reapply equation (\ref{eq:g_diatic_relation}) with $k=3$ to
%conclude that $g_{\frac{1}{8}}$ is sos, and so on.

Now, let us rewrite $g_{\lambda}$ as
\begin{equation}\nonumber
g_{\lambda}(x,y)=p(x)+\lambda(p(y)-p(x))-p(x+\lambda(y-x)).
\end{equation}
We have
\begin{equation}\label{eq:g.lambda.rewritten}
\frac{g_{\lambda}(x,y)}{\lambda}=p(y)-p(x)-\frac{p(x+\lambda(y-x))-p(x)}{\lambda}.
\end{equation}
Next, we take the limit of both sides of
(\ref{eq:g.lambda.rewritten}) by letting
$\lambda=\frac{1}{2^k}\rightarrow0$ as $k\rightarrow\infty$.
Because $p$ is differentiable, the right hand side of
(\ref{eq:g.lambda.rewritten}) will converge to $g_\nabla$. On the
other hand, our preceding argument implies that
$\frac{g_{\lambda}}{\lambda}$ is an sos polynomial (of degree $d$
in $2n$ variables) for any $\lambda=\frac{1}{2^k}$. Moreover, as
$\lambda$ goes to zero, the coefficients of
$\frac{g_{\lambda}}{\lambda}$ remain bounded since the limit of
this sequence is $g_\nabla$, which must have bounded coefficients
(see (\ref{eq:defn.g_lambda.g_grad.g_grad2})). By closedness of
the sos cone, we conclude that the limit $g_\nabla$ must be sos.

\textbf{(b)$\Rightarrow$(a)}: Assume $g_\nabla$ is sos. It is easy
to check that
\begin{equation} \nonumber
g_{\frac{1}{2}}(x,y)=\textstyle{\frac{1}{2}}g_\nabla\left(\textstyle{\frac{1}{2}}x+\textstyle{\frac{1}{2}}y,x\right)+\textstyle{\frac{1}{2}}g_\nabla\left(\textstyle{\frac{1}{2}}x+\textstyle{\frac{1}{2}}y,y\right),
\end{equation}
and hence $g_{\frac{1}{2}}$ is sos.

\textbf{(b)$\Rightarrow$(c)}: Let us write the second order Taylor
approximation of $p$ around $x$:
\begin{equation}\nonumber%\label{eq:Taylor.second.order}
\begin{array}{ll}
p(y)=&p(x)+\nabla^{T}p(x)(y-x) \\
\   &+\frac{1}{2}(y-x)^{T}H(x)(y-x)+o(||y-x||^2).
\end{array}
\end{equation}
After rearranging terms, letting $y=x+\epsilon z$ (for
$\epsilon>0$), and dividing both sides by $\epsilon^2$ we get:
\begin{equation}\label{eq:Taylor.second.order.rearranged}
%\begin{array}{ll}
(p(x+\epsilon
z)-p(x))/\epsilon^2-\nabla^{T}p(x)z/\epsilon=\frac{1}{2}z^{T}H(x)z+1/\epsilon^2o(\epsilon^2||z||^2).
%\end{array}
\end{equation}
The left hand side of (\ref{eq:Taylor.second.order.rearranged}) is
$g_\nabla(x,x+\epsilon z)/\epsilon^2$ and therefore for any fixed
$\epsilon>0$, it is an sos polynomial by assumption. As we take
$\epsilon\rightarrow0$, by closedness of the sos cone, the left
hand side of (\ref{eq:Taylor.second.order.rearranged}) converges
to an sos polynomial. On the other hand, as the limit is taken,
the term $\frac{1}{\epsilon^2}o(\epsilon^2||z||^2)$ vanishes and
hence we have that $z^TH(x)z$ must be sos.

\textbf{(c)$\Rightarrow$(b)}: Following the strategy of the proof
of the classical case in~\cite[p. 165]{Tits_lec.notes}, we start
by writing the Taylor expansion of $p$ around $x$ with the
integral form of the remainder:
\begin{equation}\label{eq:Taylor.expan.Cauchy.rem}
%\begin{array}{ll}
p(y)=p(x)+\nabla^{T}p(x)(y-x)+\int_0^1(1-t)(y-x)^{T}H(x+t(y-x))(y-x)dt.
%\end{array}
\end{equation}
Since $y^{T}H(x)y$ is sos by assumption, for any $t\in[0,1]$ the
integrand
$$(1-t)(y-x)^{T}H(x+t(y-x))(y-x)$$ is an sos polynomial of degree $d$ in $x$ and
$y$. From (\ref{eq:Taylor.expan.Cauchy.rem}) we have
$$g_\nabla=\int_0^1(1-t)(y-x)^{T}H(x+t(y-x))(y-x)dt.$$
It then follows that $g_\nabla$ is sos because integrals of sos
polynomials, if they exist, are sos. \aaa{To see the latter fact,
note that we can write the integral as a limit of a sequence of
Riemann sums by discretizing the interval $[0,1]$ over which we
are integrating. Since every finite Riemann sum is an sos
polynomial of degree $d$, and since the sos cone is closed, it
follows that the limit of the sequence must be sos.}

%This is a standard fact, but to see it intuitively one can think
%of discretizing $t$ and writing the integral as a limit of a
%sequence of Riemann sums. Every finite Riemann sum will be a
%polynomial of degree $d$ and because the sos cone is closed, the
%limit of this sequence must be sos.
%
%The conclusion that $g_\nabla$ is sos then follows by noting that
%from (\ref{eq:Taylor.expan.Cauchy.rem}) we have
%$$g_\nabla=\int_0^1(1-t)(y-x)^{T}H(x+t(y-x))(y-x)d_t,$$
%and by appealing to the fact that integrals of sos polynomials, if
%they exist, are sos. To see this fact, one can think of
%discretizing $t$ and writing the integral as a limit of a sequence
%of Riemann sums. Every finite Riemann sum will be a polynomial of
%degree $d$ and because the sos cone is closed, the limit of this
%sequence must be sos.
\end{proof}
We conclude that conditions \textbf{(a)}, \textbf{(b)}, and
\textbf{(c)} are equivalent sufficient conditions for convexity of
polynomials, and can each be checked with a semidefinite program
as explained in
Subsection~\ref{subsec:sos.sdp.and.matrix.generalize}. It is easy
to see that all three polynomials $g_{\frac{1}{2}}(x,y)$,
$g_\nabla(x,y)$, and $g_{\nabla^2}(x,y)$ are polynomials in $2n$
variables and of degree $d$. (Note that each differentiation
reduces the degree by one.) Each of these polynomials have a
specific structure that can be exploited for formulating smaller
SDPs. For example, the symmetries
$g_{\frac{1}{2}}(x,y)=g_{\frac{1}{2}}(y,x)$ and
$g_{\nabla^2}(x,-y)=g_{\nabla^2}(x,y)$ can be taken advantage of
via symmetry reduction techniques developed
in~\cite{Symmetry_groups_Gatermann_Pablo}.

The issue of symmetry reduction aside, we would like to point out
that formulation \textbf{(c)} (which was the original definition
of sos-convexity) can be significantly more efficient than the
other two conditions. The reason is that the polynomial
$g_{\nabla^2}(x,y)$ is always quadratic and homogeneous in $y$ and
of degree $d-2$ in $x$. This makes $g_{\nabla^2}(x,y)$ much more
sparse than $g_\nabla(x,y)$ and $g_{\nabla^2}(x,y)$, which have
degree $d$ both in $x$ and in $y$. Furthermore, because of the
special bipartite structure of $y^TH(x)y$, only monomials of the
form \aaa{$x^\alpha y_i$ (i.e., linear in $y$)} will appear in the
vector of monomials (\ref{eq:monomials}). This in turn reduces the
size of the Gram matrix, and hence the size of the SDP. It is
perhaps not too surprising that the characterization of convexity
based on the Hessian matrix is a more efficient condition to
check. \aaa{After all, at a given point $x$, the property of
having nonnegative curvature in every direction is a local
condition, whereas characterizations \textbf{(a)} and \textbf{(b)}
both involve global conditions.}

%After all, this is a local condition (curvature at every point in
%every direction must be nonnegative), whereas conditions
%\textbf{(a)} and \textbf{(b)} are both global.

\begin{remark}
There has been yet another proposal for an sos relaxation for
convexity of polynomials in~\cite{Chesi_Hung_journal}. However, we
have shown in~\cite{AAA_PP_CDC10_algeb_convex} that the condition
in~\cite{Chesi_Hung_journal} is at least as conservative as the
three conditions in
Theorem~\ref{thm:sos.convexity.3.equivalent.defs} and also
significantly more expensive to check.
\end{remark}

\begin{remark}\label{rmk:sos-convexity.restriction}
Just like convexity, the property of sos-convexity is preserved
under restrictions to affine subspaces. This is perhaps most
directly seen through characterization \textbf{(a)} of
sos-convexity in
Theorem~\ref{thm:sos.convexity.3.equivalent.defs}, by also noting
that sum of squares is preserved under restrictions. Unlike
convexity however, if a polynomial is sos-convex on every line (or
even on every proper affine subspace), this does not imply that
the polynomial is sos-convex.
\end{remark}

As an application of
Theorem~\ref{thm:sos.convexity.3.equivalent.defs}, we use our new
characterization of sos-convexity to give a short proof of an
interesting lemma of Helton and Nie.

%\begin{lemma}[Helton and Nie~\cite{Helton_Nie_SDP_repres_2}]\label{lem:helton.nie.sos-convex.then.sos}
%Every sos-convex form is sos.
%\end{lemma}

%AAA: When I try to put the Lemma number in Helton and Nie, [ [] ] doesn't work, so I have to cheat it.
\begin{lemma}\label{lem:helton.nie.sos-convex.then.sos}\emph{(Helton and Nie~\cite[Lemma 8]{Helton_Nie_SDP_repres_2})}.
Every sos-convex form is sos.
\end{lemma}

\begin{proof}
Let $p$ be an sos-convex form of degree $d$. We know from
Theorem~\ref{thm:sos.convexity.3.equivalent.defs} that
sos-convexity of $p$ is equivalent to the polynomial
%\begin{equation}\nonumber
$g_{\frac{1}{2}}(x,y)=\textstyle{\frac{1}{2}
p(x)}+\textstyle{\frac{1}{2}}p(y)-p\left(\textstyle{\frac{1}{2}}
x+\textstyle{\frac{1}{2}}y\right)$
%\end{equation}
being sos. But since sos is preserved under restrictions and
$p(0)=0$, this implies that
$$g_{\frac{1}{2}}(x,0)=\textstyle{\frac{1}{2}
p(x)}-p(\frac{1}{2}x)=\left(\frac{1}{2}-(\frac{1}{2})^d
\right)p(x)$$ is sos.
\end{proof}

Note that the same argument also shows that convex forms are psd.

%We shall end this section by commenting on the necessity of these
%conditions. In~\cite{AAA_PP_not_sos_convex_journal},
%\cite{AAA_PP_CDC09_HessianNotFactor}, the authors showed with an
%explicit example that there exist convex polynomials that are not
%sos-convex. In view of
%Theorems~\ref{thm:classical.first.2nd.order.charac.}~and~\ref{thm:main.result.equivalence},
%it follows immediately that the same example rejects necessity of
%conditions \textbf{(a)} and \textbf{(b)}. However, standard
%Positivstellensatz results related to Hilbert's 17th problem can
%be used to make these conditions necessary under mild assumptions
%but at the cost of solving larger and larger SDPs; see
%e.g.~\cite{Reznick_Unif_denominator}.

\section{Some constructions of convex but not sos-convex
polynomials}\label{sec:first.examples} It is natural to ask
whether sos-convexity is not only a sufficient condition for
convexity of polynomials but also a necessary one. In other words,
could it be the case that if the Hessian of a polynomial is
positive semidefinite, then it must factor? To give a negative
answer to this question, one has to prove existence of a convex
polynomial that is not sos-convex, i.e, a polynomial $p$ for which
one (and hence all) of the three polynomials $g_{\frac{1}{2}},
g_\nabla,$ and $g_{\nabla^2}$ in
(\ref{eq:defn.g_lambda.g_grad.g_grad2}) are psd but not sos. Note
that existence of psd but not sos polynomials does not imply
existence of convex but not sos-convex polynomials on its own. The
reason is that the polynomials $g_{\frac{1}{2}}, g_\nabla,$ and
$g_{\nabla^2}$ all possess a very special
structure.\footnote{There are many situations where requiring a
specific structure on polynomials makes psd equivalent to sos. As
an example, we know that there are forms in
$P_{4,4}\setminus\Sigma_{4,4}$. However, if we require the forms
to have only even monomials, then all such nonnegative forms in 4
variables and degree 4 are sums of
squares~\cite{Even_quartics_4vars_sos}.} For example, $y^TH(x)y$
has the structure of being quadratic in $y$ and a Hessian in $x$.
(Not every polynomial matrix is a valid Hessian.) The Motzkin or
the Robinson polynomials in (\ref{eq:Motzkin.form}) and
(\ref{eq:Robinston.form}) for example are clearly not of this
structure.

In an earlier paper, we presented the first example of a convex
but not sos-convex
polynomial~\cite{AAA_PP_not_sos_convex_journal},\cite{AAA_PP_CDC09_HessianNotFactor}\footnote{Assuming
P$\neq$NP, and given the NP-hardness of deciding polynomial
convexity~\cite{NPhard_Convexity_arxiv}, one would expect to see
convex polynomials that are not sos-convex. However, our first
example in~\cite{AAA_PP_not_sos_convex_journal} appeared before
the proof of NP-hardness~\cite{NPhard_Convexity_arxiv}. Moreover,
from complexity considerations, even assuming P$\neq$NP,  one
cannot conclude existence of convex but not sos-convex polynomials
for any finite value of the number of variables $n$.}:
\begin{equation}\label{eq:first.convex.not.sos.convex}
\begin{array}{rlll}
p(x_1,x_2,x_3)&=&32x_1^8+118x_1^6x_2^2+40x_1^6x_3^2+25x_1^4x_2^4-43x_1^4x_2^2x_3^2-35x_1^4x_3^4+3x_1^2x_2^4x_3^2
\\
\\ \quad&\
&-16x_1^2x_2^2x_3^4+24x_1^2x_3^6+16x_2^8+44x_2^6x_3^2+70x_2^4x_3^4+60x_2^2x_3^6+30x_3^8.
\end{array}
\end{equation}
As we will see later in this paper, this form which lives in
$C_{3,8}\setminus\Sigma C_{3,8}$ turns out to be an example in the
smallest possible number of variables but not in the smallest
degree. We next present another example of a convex but not
sos-convex form that has not been previously in print. The example
is in $C_{6,4}\setminus\Sigma C_{6,4}$ and by contrast to the
previous example, it will turn out to be minimal in the degree but
not in the number of variables. What is nice about this example is
that unlike the other examples in this paper it has not been
derived with the assistance of a computer and semidefinite
programming:

\begin{equation}\label{eq:clean.quartic.convex.not.sos.convex}
\begin{array}{rlll}
q(x_1,\ldots,x_6)&=&x_1^4+x_2^4+x_3^4+x_4^4+x_5^4+x_6^4 \\ \\
     \quad&\
&+2(x_1^2x_2^2+x_1^2x_3^2+x_2^2x_3^2
+x_4^2x_5^2+x_4^2x_6^2+x_5^2x_6^2)\\ \\
\quad&\ &     +\frac{1}{2}(x_1^2x_4^2+x_2^2x_5^2+x_3^2x_6^2)+
x_1^2x_6^2+x_2^2x_4^2+x_3^2x_5^2 \\ \\\quad&\ &
-(x_1x_2x_4x_5+x_1x_3x_4x_6+x_2x_3x_5x_6).
\end{array}
\end{equation}
The proof that this polynomial is convex but not sos-convex can be
extracted from~\cite[Thm. 2.3 and Thm.
2.5]{NPhard_Convexity_arxiv}. In there, a general procedure is
described for producing convex but not sos-convex quartic forms
from \emph{any} example of a psd but not sos biquadratic form. The
biquadratic form that has led to the form above is that of Choi
in~\cite{Choi_Biquadratic}.

Also note that the example in
(\ref{eq:clean.quartic.convex.not.sos.convex}) shows that convex
forms that possess strong symmetry properties can still fail to be
sos-convex. The symmetries in this form are inherited from the
rich symmetry structure of the biquadratic form of Choi
(see~\cite{Symmetry_groups_Gatermann_Pablo}). In general,
symmetries are of interest in the study of positive semidefinite
and sums of squares polynomials because the gap between psd and
sos can often behave very differently depending on the symmetry
properties; see e.g.~\cite{Symmetric_quartics_sos}.

%%%AAA: Pablo told me to remove the following counterexample.
%We claim that the form \scalefont{.9}
%\begin{equation}\label{eq:symnmetric.quartic.convex.not.sos.convex}
%\begin{array}{rlll}
%s(x_1,\ldots,x_6)&=&4x_4x_5x_1^2+5x_1x_2x_6^2-10x_1x_3x_4x_6+5x_3^2x_4x_5+9x_1x_3x_4^2+13x_1x_3x_5^2+3x_1x_3x_6^2
%\\
%\  &\ &\ \\
%\  &\
%&+9x_4x_6x_1^2-10x_5x_6x_1^2-11x_1x_3x_5x_6+23x_2^2x_4x_5+5x_5x_1x_2x_6-10x_2x_3x_4^2+13x_2^2x_4x_6
%\\
%\  &\ &\ \\
%\  &\
%&+5x_2x_3x_4x_5-11x_2x_3x_4x_6+13x_2^2x_5x_6+6x_3^2x_4^2+3x_1x_3x_4x_5-5x_2x_3x_5x_6+13x_2x_3x_5^2
%\\
%\  &\ &\ \\
%\  &\
%&+12x_2^2x_5^2+12x_2^2x_4^2+12x_3^2x_5^2+12x_3^2x_6^2+7x_2x_3x_6^2+3x_3^2x_4x_6+7x_3^2x_5x_6+12x_2^2x_6^2
%\\
%\  &\ &\ \\
%\  &\
%&+12x_4^2x_1^2+12x_5^2x_1^2+6x_6^2x_1^2+31x_4x_5x_1x_2+3x_4x_6x_1x_2+4x_4^2x_1x_2+23x_5^2x_1x_2
%\\
%\  &\ &\ \\
%\  &\
%&+12(x_1^4+x_2^4+x_3^4+x_4^4+x_5^4+x_6^4+x_1^2x_2^2+x_1^2x_3^2+x_2^2x_3^2+x_4^2x_5^2+x_4^2x_6^2+x_5^2x_6^2)
%\end{array}
%\end{equation} \normalsize
%is convex, not sos-convex, and satisfies the symmetry
%$$s(x_1,x_2,x_3,x_4,x_5,x_6)=s(x_4,x_5,x_6,x_1,x_2,x_3).$$ Just
%like the form $q$ in
%(\ref{eq:clean.quartic.convex.not.sos.convex}), the form $s$ in
%(\ref{eq:symnmetric.quartic.convex.not.sos.convex}) is also
%derived from a psd but not sos biquaratic form. This biquadratic
%form is a symmetric one and is presented
%in~\cite{AAA_GB_PP_Convex_ternary_quartics}.

\section{Characterization of the gap between convexity and
sos-convexity}\label{sec:full.characterization}

Now that we know there exist convex polynomials that are not
sos-convex, our final and main goal is to give a complete
characterization of the degrees and dimensions in which such
polynomials can exist. This is achieved in the next theorem.

\begin{theorem}\label{thm:full.charac.polys}
$\tilde{\Sigma C}_{n,d}=\tilde{C}_{n,d}$ if and only if $n=1$ or
$d=2$ or $(n,d)=(2,4)$.
\end{theorem}

We would also like to have such a characterization for homogeneous
polynomials. Although convexity is a property that is in some
sense more meaningful for nonhomogeneous polynomials than for
forms, one motivation for studying convexity of forms is in their
relation to norms~\cite{Blenders_Reznick}. Also, in view of the
fact that we have a characterization of the gap between
nonnegativity and sums of squares both for polynomials and for
forms, it is very natural to inquire the same result for convexity
and sos-convexity. The next theorem presents this characterization
for forms.

\begin{theorem}\label{thm:full.charac.forms}
$\Sigma C_{n,d}=C_{n,d}$ if and only if $n=2$ or $d=2$ or
$(n,d)=(3,4)$.
\end{theorem}

The result $\Sigma C_{3,4}=C_{3,4}$ of this theorem is joint work
with G. Blekherman and is to be presented in full detail
in~\cite{AAA_GB_PP_Convex_ternary_quartics}. The remainder of this
paper is solely devoted to the proof of
Theorem~\ref{thm:full.charac.polys} and the proof of
Theorem~\ref{thm:full.charac.forms} except for the case
$(n,d)=(3,4)$. Before we present these proofs, we shall make two
important remarks.

\begin{remark}\label{rmk:difficulty.homogz.dehomogz} {\bf Difficulty with homogenization and dehomogenization.}
Recall from Subsection~\ref{subsec:nonnegativity.sos.basics} and
Theorem~\ref{thm:Hilbert} that characterizing the gap between
nonnegativity and sum of squares for polynomials is equivalent to
accomplishing this task for forms. Unfortunately, the situation is
more complicated for convexity and sos-convexity and that is the
reason why we are presenting Theorems~\ref{thm:full.charac.polys}
and~\ref{thm:full.charac.forms} as separate theorems. The
difficulty arises from the fact that unlike nonnegativity and sum
of squares, convexity and sos-convexity are not always preserved
under homogenization. (Or equivalently, the properties of being
not convex and not sos-convex are not preserved under
dehomogenization.) In fact, any convex polynomial that is not psd
will no longer be convex after homogenization. This is because
convex forms are psd but the homogenization of a non-psd
polynomial is a non-psd form. Even if a convex polynomial is psd,
its homogenization may not be convex. For example the univariate
polynomial $10x_1^4-5x_1+2$ is convex and psd, but its
homogenization $10x_1^4-5x_1x_2^3+2x_2^4$ is not
convex.\footnote{What is true however is that a nonnegative form
of degree $d$ is convex if and only if the $d$-th root of its
dehomogenization is a convex function~\cite[Prop.
4.4]{Blenders_Reznick}.} To observe the same phenomenon for
sos-convexity, consider the trivariate form $p$ in
(\ref{eq:first.convex.not.sos.convex}) which is convex but not
sos-convex and define $\tilde{p}(x_2,x_3)=p(1,x_2,x_3)$. Then, one
can check that $\tilde{p}$ is sos-convex (i.e., its $2\times 2$
Hessian factors) even though its homogenization which is $p$ is
not sos-convex~\cite{AAA_PP_not_sos_convex_journal}.
\end{remark}

\begin{remark}\label{rmk:resemb.to.Hilbert} {\bf Resemblance to the result of Hilbert.} The reader
may have noticed from the statements of Theorem~\ref{thm:Hilbert}
and Theorems~\ref{thm:full.charac.polys}
and~\ref{thm:full.charac.forms} that the cases where convex
polynomials (forms) are sos-convex are exactly the same cases
where nonnegative polynomials are sums of squares! We shall
emphasize that as far as we can tell, our results do not follow
(except in the simplest cases) from Hilbert's result stated in
Theorem~\ref{thm:Hilbert}. Note that the question of convexity or
sos-convexity of a polynomial $p(x)$ in $n$ variables and degree
$d$ is about the polynomials $g_{\frac{1}{2}}(x,y),
g_\nabla(x,y),$ or $g_{\nabla^2}(x,y)$ defined in
(\ref{eq:defn.g_lambda.g_grad.g_grad2}) being psd or sos. Even
though these polynomials still have degree $d$, it is important to
keep in mind that they are polynomials \emph{in $2n$ variables}.
Therefore, there is no direct correspondence with the
characterization of Hilbert. To make this more explicit, let us
consider for example one particular claim of
Theorem~\ref{thm:full.charac.forms}: $\Sigma C_{2,4}=C_{2,4}$. For
a form $p$ in 2 variables and degree 4, the polynomials
$g_{\frac{1}{2}}, g_\nabla,$ and $g_{\nabla^2}$ will be forms in 4
variables and degree 4. We know from Hilbert's result that in this
situation psd but not sos forms do in fact exist. However, for the
forms in 4 variables and degree 4 that have the special structure
of $g_{\frac{1}{2}}, g_\nabla,$ or $g_{\nabla^2}$, psd turns out
to be equivalent to sos.
%
%Therefore, even if we take e.g. the relatively simple case of a
%homogeneous polynomial $p$ in 2 variables and degree 4, then the
%polynomials $g_{\frac{1}{2}}, g_\nabla,$ and $g_{\nabla^2}$ will
%be homogeneous polynomials in 4 variables and degree 4. Hilbert's
%result states that in this situation there are psd forms that are
%not sos. However, we prove that $\Sigma C_{2,4}=C_{2,4}$. This, of
%course, is due to the fact that there is additional structure in
%the polynomials $g_{\frac{1}{2}}, g_\nabla,$ and $g_{\nabla^2}$.
\end{remark}

The proofs of Theorems~\ref{thm:full.charac.polys}
and~\ref{thm:full.charac.forms} are broken into the next two
subsections. In Subsection~\ref{subsec:proof.equal.cases}, we
provide the proofs for the cases where convexity and sos-convexity
are equivalent. Then in
Subsection~\ref{subsec:proof.non.equal.cases}, we prove that in
all other cases there exist convex polynomials that are not
sos-convex.

\subsection{Proofs of Theorems~\ref{thm:full.charac.polys} and~\ref{thm:full.charac.forms}: cases where $\tilde{\Sigma C}_{n,d}=\tilde{C}_{n,d}, \Sigma
C_{n,d}=C_{n,d}$}\label{subsec:proof.equal.cases}

When proving equivalence of convexity and sos-convexity, it turns
out to be more convenient to work with the second order
characterization of sos-convexity, i.e., with the form
$g_{\nabla^2}(x,y)=y^TH(x)y$ in
(\ref{eq:defn.g_lambda.g_grad.g_grad2}). The reason for this is
that this form is always quadratic in $y$, and this allows us to
make use of the following key theorem, henceforth referred to as
the ``biform theorem''.

\begin{theorem}[e.g.~\cite{CLRrealzeros}]\label{thm:biform.thm}
Let $f\mathrel{\mathop:}=f(u_1,u_2,v_1,\ldots,v_m)$ be a form in
the variables $u\mathrel{\mathop:}=(u_1,u_2)^T$ and
$v\mathrel{\mathop:}=(v_1,,\ldots,v_m)^T$ that is a quadratic form
in $v$ for fixed $u$ and a form (of \aaa{any} degree) in $u$ for
fixed $v$. Then $f$ is psd if and only if it is sos.\footnote{Note
that the results $\Sigma_{2,d}=P_{2,d}$ and $\Sigma_{n,2}=P_{n,2}$
are both special cases of this theorem.}
\end{theorem}

The biform theorem has been proven independently by several
authors. See~\cite{CLRrealzeros} and~\cite{SOS_KYP} for more
background on this theorem and in particular~\cite[Sec.
7]{CLRrealzeros} for an elegant proof and some refinements. We now
proceed with our proofs which will follow in a rather
straightforward manner from the biform theorem.

\begin{theorem}
$\tilde{\Sigma C}_{1,d}=\tilde{C}_{1,d}$ for all $d$. $\Sigma
C_{2,d}=C_{2,d}$ for all $d$.
\end{theorem}

\begin{proof}
For a univariate polynomial, convexity means that the second
derivative, which is another univariate polynomial, is psd. Since
$\tilde{\Sigma}_{1,d}=\tilde{P}_{1,d}$, the second derivative must
be sos. Therefore, $\tilde{\Sigma C}_{1,d}=\tilde{C}_{1,d}$. To
prove $\Sigma C_{2,d}=C_{2,d}$, suppose we have a convex bivariate
form $p$ of degree $d$ in variables
$x\mathrel{\mathop:}=(x_1,x_2)^T$. The Hessian
$H\mathrel{\mathop:}=H(x)$ of $p$ is a $2\times 2$ matrix whose
entries are forms of degree $d-2$. If we let
$y\mathrel{\mathop:}=(y_1,y_2)^T$, convexity of $p$ implies that
the form $y^TH(x)y$ is psd. Since $y^TH(x)y$ meets the
requirements of the biform theorem above with
$(u_1,u_2)=(x_1,x_2)$ and $(v_1,v_2)=(y_1,y_2)$, it follows that
$y^TH(x)y$ is sos. Hence, $p$ is sos-convex.
\end{proof}

\begin{theorem}
$\tilde{\Sigma C}_{n,2}=\tilde{C}_{n,2}$ for all $n$. $\Sigma
C_{n,2}=C_{n,2}$ for all $n$.
\end{theorem}

\begin{proof}
Let $x\mathrel{\mathop:}=(x_1,\ldots,x_n)^T$ and
$y\mathrel{\mathop:}=(y_1,\ldots,y_n)^T$. Let
$p(x)=\frac{1}{2}x^TQx+b^Tx+c$ be a quadratic polynomial. The
Hessian of $p$ in this case is the constant symmetric matrix $Q$.
Convexity of $p$ implies that $y^TQy$ is psd. But since
$\Sigma_{n,2}=P_{n,2}$, $y^TQy$ must be sos. Hence, $p$ is
sos-convex. The proof of $\Sigma C_{n,2}=C_{n,2}$ is identical.
\end{proof}

\begin{theorem}
$\tilde{\Sigma C}_{2,4}=\tilde{C}_{2,4}$.
\end{theorem}

\begin{proof}
Let $p(x)\mathrel{\mathop:}=p(x_1,x_2)$ be a convex bivariate
quartic polynomial. Let $H\mathrel{\mathop:}=H(x)$ denote the
Hessian of $p$ and let $y\mathrel{\mathop:}=(y_1,y_2)^T$. Note
that $H(x)$ is a $2\times 2$ matrix whose entries are (not
necessarily homogeneous) quadratic polynomials. Since $p$ is
convex, $y^TH(x)y$ is psd. Let $\bar{H}(x_1,x_2,x_3)$ be a
$2\times 2$ matrix whose entries are obtained by homogenizing the
entries of $H$. It is easy to see that $y^T\bar{H}(x_1,x_2,x_3)y$
is then the form obtained by homogenizing $y^TH(x)y$ and is
therefore psd. Now we can employ the biform theorem
(Theorem~\ref{thm:biform.thm}) with $(u_1,u_2)=(y_1,y_2)$ and
$(v_1,v_2,v_3)=(x_1,x_2,x_3)$ to conclude that
$y^T\bar{H}(x_1,x_2,x_3)y$ is sos. But upon dehomogenizing by
setting $x_3=1$, we conclude that $y^TH(x)y$ is sos. Hence, $p$ is
sos-convex.
\end{proof}

\begin{theorem}[Ahmadi, Blekherman, Parrilo~\cite{AAA_GB_PP_Convex_ternary_quartics}]
$\Sigma C_{3,4}=C_{3,4}$.
\end{theorem}

Unlike Hilbert's results $\tilde{\Sigma}_{2,4}=\tilde{P}_{2,4}$
and $\Sigma_{3,4}=P_{3,4}$ which are equivalent statements and
essentially have identical proofs, the proof of $\Sigma
C_{3,4}=C_{3,4}$ is considerably more involved than the proof of
$\tilde{\Sigma C}_{2,4}=\tilde{C}_{2,4}$. Here, we briefly point
out why this is the case and refer the reader
to~\cite{AAA_GB_PP_Convex_ternary_quartics} for more details.

If $p(x)\mathrel{\mathop:}=p(x_1,x_2,x_3)$ is a ternary quartic
form, its Hessian $H(x)$ is a $3\times 3$ matrix whose entries are
quadratic forms. In this case, we can no longer apply the biform
theorem to the form $y^TH(x)y$. In fact, the matrix
\begin{equation}\nonumber
C(x)=\begin{bmatrix} x_1^2+2x_2^2&-x_1x_2&-x_1x_3 \\ \\
-x_1x_2&x_2^2+2x_3^2&-x_2x_3 \\ \\
-x_1x_3&-x_2x_3&x_3^2+2x_1^2
\end{bmatrix},
\end{equation}
due to Choi~\cite{Choi_Biquadratic} serves as an explicit example
of a $3\times 3$ matrix with quadratic form entries that is
positive semidefinite but not an sos-matrix; i.e., $y^TC(x)y$ is
psd but not sos. However, the matrix $C(x)$ above is \emph{not} a
valid Hessian, i.e., it cannot be the matrix of the second
derivatives of any polynomial. If this was the case, the third
partial derivatives would commute. On the other hand, we have in
particular
$$\frac{\partial C_{1,1}(x)}{\partial x_3}=0\neq-x_3=\frac{\partial C_{1,3}(x)}{\partial
x_1}.$$ It is rather remarkable that with the additional
requirement of being a valid Hessian, the form $y^TH(x)y$ turns
out to be psd if and only if it is
sos~\cite{AAA_GB_PP_Convex_ternary_quartics}.

\subsection{Proofs of Theorems~\ref{thm:full.charac.polys} and~\ref{thm:full.charac.forms}: cases where $\tilde{\Sigma C}_{n,d}\subset\tilde{C}_{n,d}, \Sigma
C_{n,d}\subset C_{n,d}$}\label{subsec:proof.non.equal.cases}

The goal of this subsection is to establish that the cases
presented in the previous subsection are the \emph{only} cases
where convexity and sos-convexity are equivalent. We will first
give explicit examples of convex but not sos-convex
polynomials/forms that are ``minimal'' jointly in the degree and
dimension and then present an argument for all dimensions and
degrees higher than those of the minimal cases.

\subsubsection{Minimal convex but not sos-convex polynomials/forms}
 The minimal examples of convex but not sos-convex
 polynomials (resp. forms) turn out to belong to $\tilde{C}_{2,6}\setminus\tilde{\Sigma C}_{2,6}$ and $\tilde{C}_{3,4}\setminus\tilde{\Sigma C}_{3,4}$ (resp. $C_{3,6}\setminus\Sigma C_{3,6}$ and $C_{4,4}\setminus\Sigma C_{4,4}$).
 Recall from Remark~\ref{rmk:difficulty.homogz.dehomogz} that we
 lack a general argument for going from convex but not sos-convex
 forms to polynomials or vice versa. Because of this, one would need to
 present four different polynomials in the sets mentioned above
 and prove that each polynomial is (i) convex and (ii) not
 sos-convex. This is a total of eight arguments to make which is
 quite cumbersome. However, as we will see in the proof of Theorem~\ref{thm:minimal.2.6.and.3.6} and~\ref{thm:minimal.3.4.and.4.4} below, we have been able to find examples that
 act ``nicely'' with respect to particular ways of dehomogenization. This will allow us to reduce the
 total number of claims we have to prove from eight to four.

%%%AAA: Pablo asked me to remove the following lemma.
%\begin{lemma}\label{lem:conv.sos-conv.preserv.dh}
%Let $p\mathrel{\mathop:}=p(x_1,\ldots,x_n)$ be a form. Define the
%polynomial $\tilde{p}$ as
%$$\tilde{p}(x_1,\ldots,x_{n-1})\mathrel{\mathop:}=p(x_1,\ldots,x_{n-1},a_0+a_1x_1+\cdots+a_{n-1}x_{n-1}),$$
%where $a_0,\ldots,a_{n-1}$ are some constants. If $p$ is convex,
%then so is $\tilde{p}$. If $p$ is sos-convex, then so is
%$\tilde{p}$.
%\end{lemma}
%
%\begin{proof}
%We only prove the claim about sos-convexity. The proof for
%convexity is similar and in fact standard. Let
%$x\mathrel{\mathop:}=(x_1,\ldots,x_n)^T$,
%$y\mathrel{\mathop:}=(y_1,\ldots,y_n)^T$,
%$\tilde{x}\mathrel{\mathop:}=(x_1,\ldots,x_{n-1})^T$, and
%$\tilde{y}\mathrel{\mathop:}=(y_1,\ldots,y_{n-1})^T$. If $p$ is
%sos-convex, then by
%Theorem~\ref{thm:sos.convexity.3.equivalent.defs} the form
%$$g_{\frac{1}{2}}(x,y)=\frac{1}{2}p(x)+\frac{1}{2}p(y)-p\textstyle{\left(\frac{1}{2}x+\frac{1}{2}x\right)}$$
%is sos. Let
%$$\tilde{g}_{\frac{1}{2}}(\tilde{x},\tilde{y})=\frac{1}{2}\tilde{p}(\tilde{x})+\frac{1}{2}\tilde{p}(\tilde{y})-\tilde{p}\textstyle{\left(\frac{1}{2}\tilde{x}+\frac{1}{2}\tilde{x}\right)}.$$
%It is easy to check that
%$$\tilde{g}_{\frac{1}{2}}(\tilde{x},\tilde{y})=g_{\frac{1}{2}}(\tilde{x},a_0+a_1x_1+\cdots+a_{n-1}x_{n-1},\tilde{y},a_0+a_1y_1+\cdots+a_{n-1}y_{n-1}).$$
%Therefore, $\tilde{g}_{\frac{1}{2}}$ is sos, which means that
%$\tilde{p}$ is sos-convex.
%\end{proof}

 The polynomials that we are about to present next have been
 found with the assistance of a computer and by employing some ``tricks'' with semidefinite
 programming.\footnote{We do not elaborate here on how this was exactly done. The interested reader is referred to~\cite[Sec. 4]{AAA_PP_not_sos_convex_journal}, where a similar technique for formulating semidefinite programs that can search over a convex subset of the (non-convex) set of convex but not sos-convex polynomials is explained.  The approach in~\cite{AAA_PP_not_sos_convex_journal} however does not lead to examples that are minimal.} In this process, we have made use of software
 packages YALMIP~\cite{yalmip}, SOSTOOLS~\cite{sostools}, and the
 SDP solver SeDuMi~\cite{sedumi}, which we acknowledge here. To
 make the paper relatively self-contained and to emphasize the
 fact that using \emph{rational sum of squares certificates} one
 can make such computer assisted proofs fully formal, we present the proof of Theorem~\ref{thm:minimal.2.6.and.3.6} below in the appendix. On the other hand, the proof of Theorem~\ref{thm:minimal.3.4.and.4.4}, which is very similar in style to the proof of Theorem~\ref{thm:minimal.2.6.and.3.6}, is largely omitted to save space. All of the proofs are available in electronic
 form and in their entirety at \texttt{http://aaa.lids.mit.edu/software} or at \aaan{\texttt{http://arxiv.org/abs/1111.4587}}.

\begin{theorem}\label{thm:minimal.2.6.and.3.6} $\tilde{\Sigma C}_{2,6}$ is a proper subset of
$\tilde{C}_{2,6}$. $\Sigma C_{3,6}$ is a proper subset of
$C_{3,6}$.
\end{theorem}

\begin{proof}
We claim that the form
\begin{equation}\label{eq:minim.form.3.6}
\begin{array}{lll}
f(x_1,x_2,x_3)&=&77x_1^6-155x_1^5x_2+445x_1^4x_2^2+76x_1^3x_2^3+556x_1^2x_2^4+68x_1x_2^5+240x_2^6-9x_1^5x_3 \\
\  &\ &\ \\
   \  &\ &-1129x_1^3x_2^2x_3+62x_1^2x_2^3x_3+1206x_1x_2^4x_3-343x_2^5x_3+363x_1^4x_3^2+773x_1^3x_2x_3^2 \\
   \  &\ &\ \\
    \ &\ &+891x_1^2x_2^2x_3^2-869x_1x_2^3x_3^2+1043x_2^4x_3^2-14x_1^3x_3^3-1108x_1^2x_2x_3^3-216x_1x_2^2x_3^3\\
    \  &\ &\ \\
    \ &\ &-839x_2^3x_3^3+721x_1^2x_3^4+436x_1x_2x_3^4+378x_2^2x_3^4+48x_1x_3^5-97x_2x_3^5+89x_3^6
\end{array}
\end{equation}
belongs to $C_{3,6}\setminus\Sigma C_{3,6}$, and the
polynomial\footnote{The polynomial $f(x_1,x_2,1)$ turns out to be
sos-convex, and therefore does not do the job. One can of course
change coordinates, and then in the new coordinates perform the
dehomogenization by setting $x_3=1$.}
\begin{equation}\label{eq:minim.poly.2.6}
\tilde{f}(x_1,x_2)=f(x_1,x_2,1-\frac{1}{2}x_2)
\end{equation}
belongs to $\tilde{C}_{2,6}\setminus\tilde{\Sigma C}_{2,6}$. Note
that since convexity and sos-convexity are both preserved under
restrictions to affine subspaces (recall
Remark~\ref{rmk:sos-convexity.restriction}), it suffices to show
that the form $f$ in (\ref{eq:minim.form.3.6}) is convex and the
polynomial $\tilde{f}$ in (\ref{eq:minim.poly.2.6}) is not
sos-convex. Let $x\mathrel{\mathop:}=(x_1,x_2,x_2)^T$,
$y\mathrel{\mathop:}=(y_1,y_2,y_3)^T$,
$\tilde{x}\mathrel{\mathop:}=(x_1,x_2)^T$,
$\tilde{y}\mathrel{\mathop:}=(y_1,y_2)^T$, and denote the Hessian
of $f$ and $\tilde{f}$ respectively by $H_f$ and $H_{\tilde{f}}$.
In the appendix, we provide rational Gram matrices which prove
that the form
\begin{equation}\label{eq:y.H_f.y.xi^2.3.6.example}
(x_1^2+x_2^2)\cdot y^TH_f(x)y
\end{equation}
is sos. This, together with nonnegativity of $x_1^2+x_2^2$ and
continuity of $y^TH_f(x)y$, implies that $y^TH_f(x)y$ is psd.
Therefore, $f$ is convex. The proof that $\tilde{f}$ is not
sos-convex proceeds by showing that $H_{\tilde{f}}$ is not an
sos-matrix via a separation argument. In the appendix, we present
a separating hyperplane that leaves the appropriate sos cone on
one side and the polynomial
\begin{equation}\label{eq:y.H_f_tilda.y.xi^2.2.6.example}
\tilde{y}^TH_{\tilde{f}}(\tilde{x})\tilde{y}
\end{equation}
on the other.
\end{proof}

\begin{theorem}\label{thm:minimal.3.4.and.4.4} $\tilde{\Sigma C}_{3,4}$ is a proper subset of
$\tilde{C}_{3,4}$. $\Sigma C_{4,4}$ is a proper subset of
$C_{4,4}$.
\end{theorem}

\begin{proof}
We claim that the form
\begin{equation}\label{eq:minim.form.4.4}
\begin{array}{lll}
h(x_1,\ldots,x_4)&=&1671x_1^4-4134x_1^3x_2-3332x_1^3x_3+5104x_1^2x_2^2+4989x_1^2x_2x_3+3490x_1^2x_3^2 \\
\  &\ &\ \\
 \  &\ &-2203x_1x_2^3-3030x_1x_2^2x_3-3776x_1x_2x_3^2-1522x_1x_3^3+1227x_2^4-595x_2^3x_3 \\
 \  &\ &\ \\
 \  &\ &+1859x_2^2x_3^2+1146x_2x_3^3+979x_3^4+1195728x_4^4-1932x_1x_4^3-2296x_2x_4^3 \\
 \  &\ &\ \\
 \  &\ &-3144x_3x_4^3+1465x_1^2x_4^2-1376x_1^3x_4-263x_1x_2x_4^2+2790x_1^2x_2x_4+2121x_2^2x_4^2 \\
 \  &\ &\ \\
 \  &\ &-292x_1x_2^2x_4-1224x_2^3x_4+2404x_1x_3x_4^2+2727x_2x_3x_4^2-2852x_1x_3^2x_4 \\
 \  &\ &\ \\
 \  &\ &-388x_2x_3^2x_4-1520x_3^3x_4+2943x_1^2x_3x_4-5053x_1x_2x_3x_4
 +2552x_2^2x_3x_4 \\
  \  &\ &\ \\
 \  &\ & +3512x_3^2x_4^2
\end{array}
\end{equation}
belongs to $C_{4,4}\setminus\Sigma C_{4,4}$, and the polynomial
\begin{equation}\label{eq:minim.poly.3.4}
\tilde{h}(x_1,x_2,x_3)=h(x_1,x_2,x_3,1)
\end{equation}
belongs to $\tilde{C}_{3,4}\setminus\tilde{\Sigma C}_{3,4}$. Once
again, it suffices to prove that $h$ is convex and $\tilde{h}$ is
not sos-convex. Let $x\mathrel{\mathop:}=(x_1,x_2,x_3,x_4)^T$,
$y\mathrel{\mathop:}=(y_1,y_2,y_3,y_4)^T$, and denote the Hessian
of $h$ and $\tilde{h}$ respectively by $H_h$ and $H_{\tilde{h}}$.
The proof that $h$ is convex is done by showing that the form
\begin{equation}\label{eq:y.H_h.y.xi^2.4.4.example}
(x_2^2+x_3^2+x_4^2)\cdot y^TH_h(x)y
\end{equation}
is sos.\footnote{The choice of multipliers in
(\ref{eq:y.H_f.y.xi^2.3.6.example}) and
(\ref{eq:y.H_h.y.xi^2.4.4.example}) is motivated by a result of
Reznick in~\cite{Reznick_Unif_denominator}.} The proof that
$\tilde{h}$ is not sos-convex is done again by means of a
separating hyperplane.
\end{proof}

\subsubsection{Convex but not sos-convex polynomials/forms in all higher degrees and
dimensions}\label{subsubsec:increasing.degree.and.vars} Given a
convex but not sos-convex polynomial (form) in $n$ variables, it
is very easy to argue that such a polynomial (form) must also
exist in a larger number of variables. If $p(x_1,\ldots,x_n)$ is a
form in $C_{n,d}\setminus\Sigma C_{n,d}$, then
$$\bar{p}(x_1,\ldots,x_{n+1})=p(x_1,\ldots,x_n)+x_{n+1}^d$$
belongs to $C_{n+1,d}\setminus\Sigma C_{n+1,d}$. Convexity of
$\bar{p}$ is obvious since it is a sum of convex functions. The
fact that $\bar{p}$ is not sos-convex can also easily be seen from
the block diagonal structure of the Hessian of $\bar{p}$: if the
Hessian of $\bar{p}$ were to factor, it would imply that the
Hessian of $p$ should also factor. The argument for going from
$\tilde{C}_{n,d}\setminus\tilde{\Sigma C}_{n,d}$ to
$\tilde{C}_{n+1,d}\setminus\tilde{\Sigma C}_{n+1,d}$ is identical.

Unfortunately, an argument for increasing the degree of convex but
not sos-convex forms seems to be significantly more difficult to
obtain. In fact, we have been unable to come up with a natural
operation that would produce a \aaan{form} in
$C_{n,d+2}\setminus\Sigma C_{n,d+2}$ from a form in
$C_{n,d}\setminus\Sigma C_{n,d}$. We will instead take a different
route: we are going to present a general procedure for going from
a form in $P_{n,d}\setminus\Sigma_{n,d}$ to a form in
$C_{n,d+2}\setminus\Sigma C_{n,d+2}$.  This will serve our purpose
of constructing convex but not sos-convex forms in higher degrees
and is perhaps also of independent interest in itself. For
instance, it can be used to construct convex but not sos-convex
forms that inherit structural properties (e.g. symmetry) of the
known examples of psd but not sos forms. The procedure is
constructive modulo the value of two positive constants ($\gamma$
and $\alpha$ below) whose existence will be shown
nonconstructively.\footnote{The procedure can be thought of as a
generalization of the approach in our earlier work
in~\cite{AAA_PP_not_sos_convex_journal}.}

Although the proof of the general case is no different, we present
this construction for the case $n=3$. The reason is that it
suffices for us to construct forms in $C_{3,d}\setminus\Sigma
C_{3,d}$ for $d$ even and $\geq 8$. These forms together with the
two forms in $C_{3,6}\setminus\Sigma C_{3,6}$ and
$C_{4,4}\setminus\Sigma C_{4,4}$ presented in
(\ref{eq:minim.form.3.6}) and (\ref{eq:minim.form.4.4}), and with
the simple procedure for increasing the number of variables cover
all the values of $n$ and $d$ for which convex but not sos-convex
forms exist.

For the remainder of this section, let
$x\mathrel{\mathop:}=(x_1,x_2,x_3)^T$ and
$y\mathrel{\mathop:}=(y_1,y_2,y_3)^T$.

\begin{theorem}\label{thm:conv_not_sos_conv_forms_n3d}
Let $m\mathrel{\mathop:}=m(x)$ be a ternary form of degree $d$
(with $d$ necessarily even and $\geq 6$) satisfying the following
three requirements:
%\begin{itemize}
\begin{description}
\item[R1:] $m$ is positive definite. \item[R2:] $m$ is not a sum
of squares. \item[R3:] The Hessian $H_m$ of $m$ is positive
definite at the point $(1,0,0)^T$.
%\begin{equation}\nonumber
%\Big[\frac{\partial{\ }}{\partial{x}\partial{x}} \int\int m(x)
%dx_1 dx_1 \Big] \Big\vert_{x_2=0, x_3=0} \succ 0 \quad \forall
%x_1\neq0,
%\end{equation}
%i.e., the $3\times 3$ Hessian of the form $\int\int m(x) dx_1
%dx_1$, when restricted to $x_2=0, x_3=0$, is positive definite.
\end{description}
%\end{itemize}
Let $g\mathrel{\mathop:}=g(x_2,x_3)$ be any bivariate form of
degree $d+2$ whose Hessian is positive definite. \\ Then, there
exists a constant $\gamma>0$, such that the form $f$ of degree
$d+2$ given by
\begin{equation}\label{eq:construction.of.f(x)}
f(x)=\int_0^{x_1}\int_0^s m(t,x_2,x_3) dt ds + \gamma g(x_2,x_3)
\end{equation}
is convex but not sos-convex.
\end{theorem}

\aaa{The form $f$ in (\ref{eq:construction.of.f(x)}) is just a
specific polynomial that when differentiated twice with respect to
$x_1$ gives $m$. The reason for this construction will become
clear once we present the proof of this theorem. Before we do
that, let us comment on how one can get}
%Before we prove this theorem, let us comment on how one can get
examples of forms $m$ and $g$ that satisfy the requirements of the
theorem. The choice of $g$ is in fact very easy. We can e.g. take
\begin{equation}\nonumber
g(x_2,x_3)=(x_2^2+x_3^2)^{\frac{d+2}{2}},
\end{equation}
which has a positive definite Hessian. As for the choice of $m$,
essentially any psd but not sos ternary form can be turned into a
form that satisfies requirements {\bf R1}, {\bf R2}, and {\bf R3}.
Indeed if the Hessian of such a form is positive definite at just
one point, then that point can be taken to $(1,0,0)^T$ by a change
of coordinates without changing the properties of being psd and
not sos. If the form is not positive definite, then it can made so
by adding a small enough multiple of a positive definite form to
it. For concreteness, we construct in the next lemma a family of
forms that together with the above theorem will give us convex but
not sos-convex ternary forms of any degree $\geq 8$.

\begin{lemma}\label{lem:choice.of.m(x)}
For any even degree $d\geq 6$, there exists a constant $\alpha>0$,
such that the form
\begin{equation}\label{eq:choice.of.m(x)}
m(x)=x_1^{d-6}(x_1^2x_2^4+x_1^4x_2^2-3x_1^2x_2^2x_3^2+x_3^6)+\alpha(x_1^2+x_2^2+x_3^2)^{\frac{d}{2}}
\end{equation}
satisfies the requirements {\bf R1}, {\bf R2}, and {\bf R3} of
Theorem~\ref{thm:conv_not_sos_conv_forms_n3d}.
\end{lemma}

\begin{proof}
The form $$x_1^2x_2^4+x_1^4x_2^2-3x_1^2x_2^2x_3^2+x_3^6$$ is the
familiar Motzkin form in (\ref{eq:Motzkin.form}) that is psd but
not sos~\cite{MotzkinSOS}. For any even degree $d\geq 6$, the form
$$x_1^{d-6}(x_1^2x_2^4+x_1^4x_2^2-3x_1^2x_2^2x_3^2+x_3^6)$$ is a
form of degree $d$ that is clearly still psd and less obviously
still not sos; see~\cite{Reznick}. This together with the fact
that $\Sigma_{n,d}$ is a closed cone implies existence of a small
positive value of $\alpha$ for which the form $m$ in
(\ref{eq:choice.of.m(x)}) is positive definite but not a sum of
squares, hence satisfying requirements {\bf R1} and {\bf R2}.

Our next claim is that for any positive value of $\alpha$, the
Hessian $H_m$ of the form $m$ in (\ref{eq:choice.of.m(x)})
satisfies
%\begin{equation}\label{eq:Hessian.restricted.to.x2=x3=0}
%\Big[\frac{\partial{\ }}{\partial{x}\partial{x}} \int\int m(x)
%dx_1 dx_1 \Big] \Big\vert_{x_2=0, x_3=0}=\begin{bmatrix} c_1x_1^d
%& 0 & 0 \\ 0 & c_2x_1^d & 0 \\ 0& 0& c_3x_1^d
%\end{bmatrix}
%\end{equation}
\begin{equation}\label{eq:Hessian.at.1.0.0.}
H_m(1,0,0)=\begin{bmatrix} c_1 & 0 & 0 \\ 0 & c_2 & 0 \\
0& 0& c_3
\end{bmatrix}
\end{equation}
for some positive constants $c_1,c_2,c_3$, therefore also passing
requirement {\bf R3}. To see the above equality, first note that
since $m$ is a form of degree $d$, its Hessian $H_m$ will have
entries that are forms of degree $d-2$. Therefore, the only
monomials that can survive in this Hessian after setting $x_2$ and
$x_3$ to zero are multiples of $x_1^{d-2}$. It is easy to see that
an $x_1^{d-2}$ monomial in an off-diagonal entry of $H_m$ would
lead to a monomial in $m$ that is not even. On the other hand, the
form $m$ in (\ref{eq:choice.of.m(x)}) only has even monomials.
This explains why the off-diagonal entries of the right hand side
of (\ref{eq:Hessian.at.1.0.0.}) are zero. Finally, we note that
for any positive value of $\alpha$, the form $m$ in
(\ref{eq:choice.of.m(x)}) includes positive multiples of $x_1^d$,
$x_1^{d-2}x_2^2$, and $x_1^{d-2}x_3^2$, which lead to positive
multiples of $x_1^{d-2}$ on the diagonal of $H_m$. Hence, $c_1,
c_2$, and $c_3$ are positive.
\end{proof}

Next, we state two lemmas that will be employed in the proof of
Theorem~\ref{thm:conv_not_sos_conv_forms_n3d}.

\begin{lemma}\label{lem:y.Hm.y>0.on.some.part}
Let $m$ be a trivariate form satisfying the requirements {\bf R1}
and {\bf R3} of Theorem~\ref{thm:conv_not_sos_conv_forms_n3d}. Let
$H_{\hat{m}}$ denote the Hessian of the form $\int_0^{x_1}\int_0^s
m(t,x_2,x_3) dt ds$. Then, there exists a positive constant
$\delta,$ such that
$$y^TH_{\hat{m}}(x)y>0$$ on the set
\begin{equation}\label{eq:the.set.S}
\mathcal{S}\mathrel{\mathop:}=\{(x,y) \ |\  ||x||=1, ||y||=1, \
(x_2^2+x_3^2<\delta \ \mbox{or}\ y_2^2+y_3^2<\delta)\}.
\end{equation}
\end{lemma}

\begin{proof}
We observe that when $y_2^2+y_3^2=0$, we have
$$y^TH_{\hat{m}}(x)y=y_1^2m(x),$$ which by requirement {\bf R1} is
positive when $||x||=||y||=1$. By continuity of the form
$y^TH_{\hat{m}}(x)y$, we conclude that there exists a small
positive constant $\delta_y$ such that $y^TH_{\hat{m}}(x)y>0$ on
the set
$$\mathcal{S}_y\mathrel{\mathop:}=\{(x,y) \ |\  ||x||=1, ||y||=1, \ y_2^2+y_3^2<\delta_y\}.$$
Next, we leave it to the reader to check that
$$H_{\hat{m}}(1,0,0)=\frac{1}{d(d-1)}H_m(1,0,0).$$ Therefore, when $x_2^2+x_3^2=0$, requirement {\bf R3} implies that
$y^TH_{\hat{m}}(x)y$ is positive when $||x||=||y||=~1$. Appealing
to continuity again, we conclude that there exists a small
positive constant $\delta_x$ such that $y^TH_{\hat{m}}(x)y>0$ on
the set
$$\mathcal{S}_x\mathrel{\mathop:}=\{(x,y) \ |\  ||x||=1, ||y||=1, \
x_2^2+x_3^2<\delta_x\}.$$ If we now take
$\delta=\min\{\delta_y,\delta_x\}$, the lemma is established.
\end{proof}

The last lemma that we need has already been proven in our earlier
work in~\cite{AAA_PP_not_sos_convex_journal}.

\begin{lemma}[\cite{AAA_PP_not_sos_convex_journal}]\label{lem:sos.mat.then.minor.sos}
All principal minors of an sos-matrix are sos
polynomials.\footnote{As a side note, we remark that the converse
of Lemma~\ref{lem:sos.mat.then.minor.sos} is not true even for
polynomial matrices that are valid Hessians. For example, all 7
principal minors of the $3\times 3$ Hessian of the form $f$ in
(\ref{eq:minim.form.3.6}) are sos polynomials, even though this
Hessian is not an sos-matrix.}
\end{lemma}

We are now ready to prove
Theorem~\ref{thm:conv_not_sos_conv_forms_n3d}.

\begin{proof}[Proof of Theorem~\ref{thm:conv_not_sos_conv_forms_n3d}]
We first prove that the form $f$ in
(\ref{eq:construction.of.f(x)}) is not sos-convex. By
Lemma~\ref{lem:sos.mat.then.minor.sos}, if $f$ was sos-convex,
then all diagonal elements of its Hessian would have to be sos
polynomials. On the other hand, we have from
(\ref{eq:construction.of.f(x)}) that
$$\frac{\partial{f(x)}}{\partial{x_1}\partial{x_1}}=m(x),$$ which by requirement {\bf R2} is not sos. Therefore $f$ is not sos-convex.

It remains to show that there exists a positive value of $\gamma$
for which $f$ becomes convex. Let us denote the Hessians of $f$,
$\int_0^{x_1}\int_0^s m(t,x_2,x_3) dt ds$, and $g$, by $H_f$,
$H_{\hat{m}}$, and $H_g$ respectively. So, we have
$$H_f(x)=H_{\hat{m}}(x)+\gamma H_g(x_2,x_3).$$ (Here, $H_g$ is a
$3\times 3$ matrix whose first row and column are zeros.)
Convexity of $f$ is of course equivalent to nonnegativity of the
form $y^TH_f(x)y$. Since this form is bi-homogeneous in $x$ and
$y$, it is nonnegative if and only if $y^TH_f(x)y\geq0$ on the
bi-sphere
$$\mathcal{B}\mathrel{\mathop:}=\{(x,y) \ | \ ||x||=1,
||y||=1\}.$$ Let us decompose the bi-sphere as
$$\mathcal{B}=\mathcal{S}\cup\bar{\mathcal{S}},$$
where $\mathcal{S}$ is defined in (\ref{eq:the.set.S}) and
\begin{equation}\nonumber
\bar{\mathcal{S}}\mathrel{\mathop:}=\{(x,y) \ |\  ||x||=1,
||y||=1, x_2^2+x_3^2\geq\delta,  y_2^2+y_3^2\geq\delta\}.
\end{equation}
Lemma~\ref{lem:y.Hm.y>0.on.some.part} together with positive
definiteness of $H_g$ imply that $y^TH_f(x)y$ is positive on
$\mathcal{S}$. As for the set $\bar{\mathcal{S}}$, let
\begin{equation}\nonumber
\beta_1=\min_{x,y,\in\bar{\mathcal{S}}} y^TH_{\hat{m}}(x)y,
\end{equation}
and
\begin{equation}\nonumber
\beta_2=\min_{x,y,\in\bar{\mathcal{S}}} y^TH_g(x_2,x_3)y.
\end{equation}
By the assumption of positive definiteness of $H_g$, we have
$\beta_2>0$. If we now let $$\gamma>\frac{|\beta_1|}{\beta_2},$$
then $$\min_{x,y,\in\bar{\mathcal{S}}}
y^TH_f(x)y>\beta_1+\frac{|\beta_1|}{\beta_2}\beta_2\geq0.$$ Hence
$y^TH_f(x)y$ is nonnegative (in fact positive) everywhere on
$\mathcal{B}$ and the proof is completed.
\end{proof}

Finally, we provide an argument for existence of bivariate
polynomials of degree $8,10,12,\ldots$ that are convex but not
sos-convex.

\begin{corollary}\label{cor:bivariate.polys.8.10.12...}
Consider the form $f$ in (\ref{eq:construction.of.f(x)})
constructed as described in
Theorem~\ref{thm:conv_not_sos_conv_forms_n3d}. Let
$$\tilde{f}(x_1,x_2)=f(x_1,x_2,1).$$
Then, $\tilde{f}$ is convex but not sos-convex.
\end{corollary}

\begin{proof}
The polynomial $\tilde{f}$ is convex because it is the restriction
of a convex function. It is not difficult to see that
$$\frac{\partial{\tilde{f}}(x_1,x_2)}{\partial{x_1}\partial{x_1}}=m(x_1,x_2,1),$$
which is not sos. Therefore from
Lemma~\ref{lem:sos.mat.then.minor.sos} $\tilde{f}$ is not
sos-convex.
\end{proof}

Corollary~\ref{cor:bivariate.polys.8.10.12...} together with the
two polynomials in $\tilde{C}_{2,6}\setminus\tilde{\Sigma
C}_{2,6}$ and $\tilde{C}_{3,4}\setminus\tilde{\Sigma C}_{3,4}$
presented in (\ref{eq:minim.poly.2.6}) and
(\ref{eq:minim.poly.3.4}), and with the simple procedure for
increasing the number of variables described at the beginning of
Subsection~\ref{subsubsec:increasing.degree.and.vars} cover all
the values of $n$ and $d$ for which convex but not sos-convex
polynomials exist.

\section{Concluding remarks and an open problem}\label{sec:concluding.remarks}

To conclude our paper, we would like to point out some
similarities between nonnegativity and convexity that deserve
attention: (i) both nonnegativity and convexity are properties
that only hold for even degree polynomials, (ii) for quadratic
forms, nonnegativity is in fact equivalent to convexity, (iii)
both notions are NP-hard to check exactly for degree 4 and larger,
and most strikingly (iv) nonnegativity is equivalent to sum of
squares \emph{exactly} in dimensions and degrees where convexity
is equivalent to sos-convexity. It is unclear to us whether there
can be a deeper and more unifying reason explaining these
observations, in particular, the last one which was the main
result of this paper.

Another intriguing question is to investigate whether one can give
a direct argument proving the fact that $\tilde{\Sigma
C}_{n,d}=\tilde{C}_{n,d}$ if and only if $\Sigma
C_{n+1,d}=C_{n+1,d}$. This would eliminate the need for studying
polynomials and forms separately, and in particular would provide
a short proof of the result $\Sigma_{3,4}=C_{3,4}$ given
in~\cite{AAA_GB_PP_Convex_ternary_quartics}.

Finally, an open problem related to this work is to \emph{find an
explicit example of a convex form that is not a sum of squares}.
Blekherman~\cite{Blekherman_convex_not_sos} has shown via volume
arguments that for degree $d\geq 4$ and asymptotically for large
$n$ such forms must exist, although no examples are known. In
particular, it would \aaa{be} interesting to determine the
smallest value of $n$ for which such a form exists. We know from
Lemma~\ref{lem:helton.nie.sos-convex.then.sos} that a convex form
that is not sos must necessarily be not sos-convex. Although our
several constructions of convex but not sos-convex polynomials
pass this necessary condition, the polynomials themselves are all
sos. The question is particularly interesting from an optimization
viewpoint because it implies that the well-known sum of squares
relaxation for minimizing
polynomials~\cite{Shor},~\cite{Minimize_poly_Pablo} \aaa{is not
always} exact even for the easy case of minimizing convex
polynomials.

\bibliographystyle{abbrv}
\bibliography{pablo_amirali}

\def\cprime{$'$}
\begin{thebibliography}{10}

\bibitem{AAA_GB_PP_Convex_ternary_quartics}
A.~A. Ahmadi, G.~Blekherman, and P.~A. Parrilo.
\newblock Convex ternary quartics are sos-convex.
\newblock In preparation, 2011.

\bibitem{NPhard_Convexity_arxiv}
A.~A. Ahmadi, A.~Olshevsky, P.~A. Parrilo, and J.~N. Tsitsiklis.
\newblock {NP}-hardness of deciding convexity of quartic polynomials and
  related problems.
\newblock {\em Mathematical Programming}, 2011.
\newblock DOI: 10.1007/s10107-011-0499-2.

\bibitem{AAA_PP_CDC09_HessianNotFactor}
A.~A. Ahmadi and P.~A. Parrilo.
\newblock A positive definite polynomial {H}essian that does not factor.
\newblock In {\em Proceedings of the 48$^{th}$ IEEE Conference on Decision and
  Control}, 2009.

\bibitem{AAA_PP_CDC10_algeb_convex}
A.~A. Ahmadi and P.~A. Parrilo.
\newblock On the equivalence of algebraic conditions for convexity and
  quasiconvexity of polynomials.
\newblock In {\em Proceedings of the 49$^{th}$ IEEE Conference on Decision and
  Control}, 2010.

\bibitem{AAA_PP_not_sos_convex_journal}
A.~A. Ahmadi and P.~A. Parrilo.
\newblock A convex polynomial that is not sos-convex.
\newblock {\em Mathematical Programming}, 2011.
\newblock DOI: 10.1007/s10107-011-0457-z.

\bibitem{SOS_KYP}
E.~M. Aylward, S.~M. Itani, and P.~A. Parrilo.
\newblock Explicit {SOS} decomposition of univariate polynomial matrices and
  the {K}alman-{Y}akubovich-{P}opov lemma.
\newblock In {\em Proceedings of the 46$^{th}$ IEEE Conference on Decision and
  Control}, 2007.

\bibitem{Blekherman_convex_not_sos}
G.~Blekherman.
\newblock Convex forms that are not sums of squares.
\newblock arXiv:0910.0656., 2009.

\bibitem{Blekherman_nonnegative_and_sos}
G.~Blekherman.
\newblock Nonnegative polynomials and sums of squares.
\newblock To appear in the {J}ournal of the {AMS}. Online version available at
  arXiv:1010.3465, 2010.

\bibitem{Symmetric_quartics_sos}
G.~Blekherman and C.~B. Riener.
\newblock Symmetric nonnegative forms and sums of squares.
\newblock arXiv:1205.3102, 2012.

\bibitem{BoydBook}
S.~Boyd and L.~Vandenberghe.
\newblock {\em Convex Optimization}.
\newblock Cambridge University Press, 2004.

\bibitem{Chesi_Hung_journal}
G.~Chesi and Y.~S. Hung.
\newblock Establishing convexity of polynomial {L}yapunov functions and their
  sublevel sets.
\newblock {\em IEEE Trans. Automat. Control}, 53(10):2431--2436, 2008.

\bibitem{Choi_Biquadratic}
M.~D. Choi.
\newblock Positive semidefinite biquadratic forms.
\newblock {\em Linear Algebra and its Applications}, 12:95--100, 1975.

\bibitem{Choi_Lam_extremalPSDforms}
M.~D. Choi and T.~Y. Lam.
\newblock Extremal positive semidefinite forms.
\newblock {\em Math. Ann.}, 231:1--18, 1977.

\bibitem{CLRrealzeros}
M.~D. Choi, T.-Y. Lam, and B.~Reznick.
\newblock Real zeros of positive semidefinite forms. {I}.
\newblock {\em Math. Z.}, 171(1):1--26, 1980.

\bibitem{Monique_Etienne_Convex}
E.~de~Klerk and M.~Laurent.
\newblock On the {L}asserre hierarchy of semidefinite programming relaxations
  of convex polynomial optimization problems.
\newblock {\em SIAM Journal on Optimization}, 21:824--832, 2011.

\bibitem{Even_quartics_4vars_sos}
P.~H. Diananda.
\newblock On non-negative forms in real variables some or all of which are
  non-negative.
\newblock {\em Proceedings of the {C}ambridge {P}hilosophical {S}ociety},
  58:17--25, 1962.

\bibitem{Pablo_Sep_Entang_States}
A.~C. Doherty, P.~A. Parrilo, and F.~M. Spedalieri.
\newblock Distinguishing separable and entangled states.
\newblock {\em Physical {R}eview {L}etters}, 88(18), 2002.

\bibitem{Symmetry_groups_Gatermann_Pablo}
K.~Gatermann and P.~A. Parrilo.
\newblock Symmetry groups, semidefinite programs, and sums of squares.
\newblock {\em Journal of Pure and Applied Algebra}, 192:95--128, 2004.

\bibitem{Stability_number_SOS}
N.~Gvozdenovi\'c and M.~Laurent.
\newblock Semidefinite bounds for the stability number of a graph via sums of
  squares of polynomials.
\newblock {\em Mathematical Programming}, 110(1):145--173, 2007.

\bibitem{Helton_Nie_SDP_repres_2}
J.~W. Helton and J.~Nie.
\newblock Semidefinite representation of convex sets.
\newblock {\em Mathematical Programming}, 122(1, Ser. A):21--64, 2010.

\bibitem{PositivePolyInControlBook}
D.~Henrion and A.~Garulli, editors.
\newblock {\em Positive polynomials in control}, volume 312 of {\em Lecture
  Notes in Control and Information Sciences}.
\newblock Springer, 2005.

\bibitem{Hilbert_1888}
D.~Hilbert.
\newblock \"{U}ber die {D}arstellung {D}efiniter {F}ormen als {S}umme von
  {F}ormenquadraten.
\newblock {\em Math. Ann.}, 32, 1888.

\bibitem{Kojima_SOS_matrix}
M.~Kojima.
\newblock Sums of squares relaxations of polynomial semidefinite programs.
\newblock {\em Research report B-397, Dept. of Mathematical and Computing
  Sciences. Tokyo Institute of Technology}, 2003.

\bibitem{Lasserre_Jensen_inequality}
J.~B. Lasserre.
\newblock Convexity in semialgebraic geometry and polynomial optimization.
\newblock {\em SIAM Journal on Optimization}, 19(4):1995--2014, 2008.

\bibitem{Lasserre_Convex_Positive}
J.~B. Lasserre.
\newblock {Representation of nonnegative convex polynomials}.
\newblock {\em Archiv der Mathematik}, 91(2):126--130, 2008.

\bibitem{Lasserre_set_convexity}
J.~B. Lasserre.
\newblock Certificates of convexity for basic semi-algebraic sets.
\newblock {\em Applied Mathematics Letters}, 23(8):912--916, 2010.

\bibitem{yalmip}
J.~L\"ofberg.
\newblock Yalmip : A toolbox for modeling and optimization in {MATLAB}.
\newblock In {\em Proceedings of the CACSD Conference}, 2004.
\newblock Available from
  \texttt{http://control.ee.ethz.ch/\~{}joloef/yalmip.php}.

\bibitem{convex_fitting}
A.~Magnani, S.~Lall, and S.~Boyd.
\newblock Tractable fitting with convex polynomials via sum of squares.
\newblock In {\em Proceedings of the 44$^{th}$ IEEE Conference on Decision and
  Control}, 2005.

\bibitem{MotzkinSOS}
T.~S. Motzkin.
\newblock The arithmetic-geometric inequality.
\newblock In {\em Inequalities ({P}roc. {S}ympos. {W}right-{P}atterson {A}ir
  {F}orce {B}ase, {O}hio, 1965)}, pages 205--224. Academic Press, New York,
  1967.

\bibitem{nonnegativity_NP_hard}
K.~G. Murty and S.~N. Kabadi.
\newblock Some {NP}-complete problems in quadratic and nonlinear programming.
\newblock {\em Mathematical Programming}, 39:117--129, 1987.

\bibitem{open_complexity}
P.~M. Pardalos and S.~A. Vavasis.
\newblock Open questions in complexity theory for numerical optimization.
\newblock {\em Mathematical Programming}, 57(2):337--339, 1992.

\bibitem{PhD:Parrilo}
P.~A. Parrilo.
\newblock {\em Structured semidefinite programs and semialgebraic geometry
  methods in robustness and optimization}.
\newblock PhD thesis, California Institute of Technology, May 2000.

\bibitem{sdprelax}
P.~A. Parrilo.
\newblock Semidefinite programming relaxations for semialgebraic problems.
\newblock {\em Mathematical Programming}, 96(2, Ser. B):293--320, 2003.

\bibitem{Pablo_poly_games}
P.~A. Parrilo.
\newblock Polynomial games and sum of squares optimization.
\newblock In {\em Proceedings of the 45$^{th}$ IEEE Conference on Decision and
  Control}, 2006.

\bibitem{Minimize_poly_Pablo}
P.~A. Parrilo and B.~Sturmfels.
\newblock Minimizing polynomial functions.
\newblock {\em Algorithmic and Quantitative Real Algebraic Geometry, DIMACS
  Series in Discrete Mathematics and Theoretical Computer Science}, 60:83--99,
  2003.

\bibitem{Scheiderer_ternary_quartic}
A.~Pfister and C.~Scheiderer.
\newblock An elementary proof of {H}ilbert's theorem on ternary quartics.
\newblock arXiv:1009.3144, 2010.

\bibitem{NewApproach_Hilbert_Ternary_Quatrics}
V.~Powers, B.~Reznick, C.~Scheiderer, and F.~Sottile.
\newblock A new approach to {H}ilbert's theorem on ternary quartics.
\newblock {\em Comptes Rendus Mathematique}, 339(9):617 -- 620, 2004.

\bibitem{sostools}
S.~Prajna, A.~Papachristodoulou, and P.~A. Parrilo.
\newblock {\em {SOSTOOLS}: Sum of squares optimization toolbox for {MATLAB}},
  2002-05.
\newblock Available from \texttt{http://www.cds.caltech.edu/sostools} and
  \texttt{http://www.mit.edu/\~{}parrilo/sostools}.

\bibitem{Reznick_Unif_denominator}
B.~Reznick.
\newblock Uniform denominators in {H}ilbert's 17th problem.
\newblock {\em Math Z.}, 220(1):75--97, 1995.

\bibitem{Reznick}
B.~Reznick.
\newblock Some concrete aspects of {H}ilbert's 17th problem.
\newblock In {\em Contemporary Mathematics}, volume 253, pages 251--272.
  American Mathematical Society, 2000.

\bibitem{Reznick_Hilbert_construciton}
B.~Reznick.
\newblock On {H}ilbert's construction of positive polynomials.
\newblock arXiv:0707.2156., 2007.

\bibitem{Blenders_Reznick}
B.~Reznick.
\newblock Blenders.
\newblock arXiv:1008.4533, 2010.

\bibitem{RobinsonSOS}
R.~M. Robinson.
\newblock Some definite polynomials which are not sums of squares of real
  polynomials.
\newblock In {\em Selected questions of algebra and logic (collection dedicated
  to the memory of {A}. {I}. {M}al\cprime cev) ({R}ussian)}, pages 264--282.
  Izdat. ``Nauka'' Sibirsk. Otdel., Novosibirsk, 1973.

\bibitem{Rudin_RealComplexAnalysis}
W.~Rudin.
\newblock {\em Real and complex analysis}.
\newblock Mc{G}raw-{H}ill series in higher mathematics, 1987.
\newblock Third edition.

\bibitem{matrix_sos_Hol}
C.~W. Scherer and C.~W.~J. Hol.
\newblock Matrix sum of squares relaxations for robust semidefinite programs.
\newblock {\em Mathematical Programming}, 107:189--211, 2006.

\bibitem{Shor}
N.~Z. Shor.
\newblock Class of global minimum bounds of polynomial functions.
\newblock {\em Cybernetics}, 23(6):731--734, 1987.
\newblock (Russian orig.: Kibernetika, No. 6, (1987), 9--11).

\bibitem{sedumi}
J.~Sturm.
\newblock {\em {SeDuMi} version 1.05}, Oct. 2001.
\newblock Latest version available at \texttt{http://sedumi.ie.lehigh.edu/}.

\bibitem{Tits_lec.notes}
A.~L. Tits.
\newblock Lecture notes on optimal control.
\newblock Available from \texttt{http://www.isr.umd.edu/\~{}andre/664.pdf},
  2008.

\bibitem{VaB:96}
L.~Vandenberghe and S.~Boyd.
\newblock Semidefinite programming.
\newblock {\em SIAM Review}, 38(1):49--95, Mar. 1996.

\end{thebibliography}

\appendix

\section{Certificates complementing the proof of Theorem~\ref{thm:minimal.2.6.and.3.6}}
Let $x\mathrel{\mathop:}=(x_1,x_2,x_2)^T$,
$y\mathrel{\mathop:}=(y_1,y_2,y_3)^T$,
$\tilde{x}\mathrel{\mathop:}=(x_1,x_2)^T$,
$\tilde{y}\mathrel{\mathop:}=(y_1,y_2)^T$, and let $f, \tilde{f},
H_f,$ and $H_{\tilde{f}}$ be as in the proof of
Theorem~\ref{thm:minimal.2.6.and.3.6}. This appendix proves that
the form $(x_1^2+x_2^2)\cdot y^TH_f(x)y$ in
(\ref{eq:y.H_f.y.xi^2.3.6.example}) is sos and that the polynomial
$\tilde{y}^TH_{\tilde{f}}(\tilde{x})\tilde{y}$ in
(\ref{eq:y.H_f_tilda.y.xi^2.2.6.example}) is not sos, hence
proving respectively that $f$ is convex and $\tilde{f}$ is not
sos-convex.

A rational sos decomposition of $(x_1^2+x_2^2)\cdot y^TH_f(x)y$,
which is a form in $6$ variables of degree $8$, is as follows:
\begin{equation}\nonumber
(x_1^2+x_2^2)\cdot y^TH_f(x)y=\frac{1}{84}z^TQz,
\end{equation}
where $z$ is the vector of monomials
\begin{equation}\nonumber
\begin{array}{ll}
z=&[  x_2x_3^2y_3,
  x_2x_3^2y_2,
  x_2x_3^2y_1,
  x_2^2x_3y_3,
  x_2^2x_3y_2,
  x_2^2x_3y_1,
     x_2^3y_3,
     x_2^3y_2,
     x_2^3y_1,\\
\ & \ \\
\ &    x_1x_3^2y_3,
  x_1x_3^2y_2,
  x_1x_3^2y_1,
 x_1x_2x_3y_3,
 x_1x_2x_3y_2,
 x_1x_2x_3y_1,
  x_1x_2^2y_3,
  x_1x_2^2y_2,
  x_1x_2^2y_1,\\
\ &  \ \\
\ &   x_1^2x_3y_3,
  x_1^2x_3y_2,
  x_1^2x_3y_1,
  x_1^2x_2y_3,
  x_1^2x_2y_2,
  x_1^2x_2y_1,
     x_1^3y_3,
     x_1^3y_2,
     x_1^3y_1]^T,
\end{array}
\end{equation}
and $Q$ is the $27\times 27$ positive definite
matrix\footnote{Whenever we state a matrix is positive definite,
this claim is backed up by a rational $LDL^T$ factorization of the
matrix that the reader can find online at
\texttt{http://aaa.lids.mit.edu/software} or at
\aaan{\texttt{http://arxiv.org/abs/1111.4587}}.} presented on the
next page
\begin{equation}\nonumber
Q=\begin{bmatrix} Q_1 & Q_2
\end{bmatrix},
\end{equation}

\newpage
\begin{landscape}
\setcounter{MaxMatrixCols}{30} \scalefont{.6}
\begin{equation}\nonumber
Q_1=\begin{bmatrix}[r]224280  & -40740  & 20160  & -81480  & 139692  & 93576  & -50540  & -27804  & -48384  & 0  & -29400  & -32172  & 21252  & 103404   \\
-40740  & 63504  & 36624  & 114324  & -211428  & -8316  & -67704  & 15372  & -47376  & 29400  & 0  & 16632  & 75936  & 15540   \\
20160  & 36624  & 121128  & 52920  & -27972  & -93072  & -42252  & -77196  & -58380  & 32172  & -16632  & 0  & 214284  & -57960   \\
-81480  & 114324  & 52920  & 482104  & -538776  & 36204  & -211428  & 362880  & -70644  & 19068  & 13524  & -588  & 179564  & 20258   \\
139692  & -211428  & -27972  & -538776  & 1020600  & -94416  & 338016  & -288120  & 188748  & -46368  & -33684  & -46620  & -113638  & -119112   \\
93576  & -8316  & -93072  & 36204  & -94416  & 266448  & -75348  & 216468  & 5208  & 3360  & -33432  & -31080  & -221606  & 254534   \\
-50540  & -67704  & -42252  & -211428  & 338016  & -75348  & 175224  & -144060  & 101304  & -20692  & 7826  & -4298  & -77280  & -108192   \\
-27804  & 15372  & -77196  & 362880  & -288120  & 216468  & -144060  & 604800  & 28560  & -35350  & -840  & -35434  & -132804  & 134736   \\
-48384  & -47376  & -58380  & -70644  & 188748  & 5208  & 101304  & 28560  & 93408  & -21098  & 2786  & -11088  & -104496  & -22680   \\
0  & 29400  & 32172  & 19068  & -46368  & 3360  & -20692  & -35350  & -21098  & 224280  & -40740  & 20160  & 35028  & 89964   \\
-29400  & 0  & -16632  & 13524  & -33684  & -33432  & 7826  & -840  & 2786  & -40740  & 63504  & 36624  & 51828  & -196476   \\
-32172  & 16632  & 0  & -588  & -46620  & -31080  & -4298  & -35434  & -11088  & 20160  & 36624  & 121128  & 29148  & -9408   \\
21252  & 75936  & 214284  & 179564  & -113638  & -221606  & -77280  & -132804  & -104496  & 35028  & 51828  & 29148  & 782976  & -463344   \\
103404  & 15540  & -57960  & 20258  & -119112  & 254534  & -108192  & 134736  & -22680  & 89964  & -196476  & -9408  & -463344  & 1167624   \\
267456  & -48132  & 27552  & -49742  & 78470  & 124236  & -100464  & 61404  & -90384  & 55524  & -50064  & -145908  & 41016  & -15456   \\
60872  & -25690  & -142478  & 22848  & -113820  & 259980  & -72996  & 237972  & 20412  & -95580  & -47964  & -27780  & -438732  & 514500   \\
37730  & -99036  & -10150  & -83160  & 473088  & 34188  & 167244  & 57120  & 159264  & 10752  & -93048  & -183540  & 230832  & -49980   \\
-119210  & 9170  & 81648  & 244356  & -41664  & -194124  & -9996  & 214368  & 19152  & -89184  & 2940  & -48480  & 204708  & -85344   \\
-116508  & 81564  & 26124  & 155832  & -308280  & -78180  & -74088  & 14616  & -49644  & 40320  & 87108  & 225456  & 135744  & 8568   \\
30660  & -14952  & 11844  & -21420  & 62604  & 14364  & 13608  & 1176  & 5124  & 59388  & -18144  & -99624  & 31332  & -178248   \\
35700  & 11340  & 52836  & -70788  & 86184  & 9396  & 12264  & -108024  & -11256  & 259056  & -86520  & -3528  & -19334  & 142128   \\
102156  & 2856  & 64536  & 22176  & -4200  & 77532  & -70896  & 54348  & -49616  & 72744  & -78876  & -144998  & 29316  & 23856   \\
-86100  & 47148  & 71820  & 230916  & -223692  & -131628  & -72156  & 59640  & -31416  & -75096  & -39396  & -44520  & 158508  & 308196   \\
95364  & -504  & -8412  & -23100  & 28140  & 81648  & -26768  & -25200  & -13944  & -51002  & -39228  & 71232  & 130298  & 298956   \\
11256  & 5208  & 32158  & -33264  & 45444  & 3122  & 6888  & -34440  & -5628  & 61320  & -19152  & 8988  & 18060  & -19467   \\
0  & -1344  & -3696  & -34692  & 33768  & 5964  & 9492  & -20244  & 5208  & -30072  & -9912  & 58884  & -50883  & 151956   \\
-51422  & 49056  & 32592  & 160370  & -229068  & -36792  & -68796
& 57708  & -39564  & 55944  & 31164  & -8008  & 141876  & -126483
\end{bmatrix},
\end{equation}
\normalsize

\setcounter{MaxMatrixCols}{30} \scalefont{.6}
\begin{equation}\nonumber
Q_2=\begin{bmatrix}[r]267456  & 60872  & 37730  & -119210  & -116508  & 30660  & 35700  & 102156  & -86100  & 95364  & 11256  & 0  & -51422   \\
-48132  & -25690  & -99036  & 9170  & 81564  & -14952  & 11340  & 2856  & 47148  & -504  & 5208  & -1344  & 49056   \\
27552  & -142478  & -10150  & 81648  & 26124  & 11844  & 52836  & 64536  & 71820  & -8412  & 32158  & -3696  & 32592   \\
-49742  & 22848  & -83160  & 244356  & 155832  & -21420  & -70788  & 22176  & 230916  & -23100  & -33264  & -34692  & 160370   \\
78470  & -113820  & 473088  & -41664  & -308280  & 62604  & 86184  & -4200  & -223692  & 28140  & 45444  & 33768  & -229068   \\
124236  & 259980  & 34188  & -194124  & -78180  & 14364  & 9396  & 77532  & -131628  & 81648  & 3122  & 5964  & -36792   \\
-100464  & -72996  & 167244  & -9996  & -74088  & 13608  & 12264  & -70896  & -72156  & -26768  & 6888  & 9492  & -68796   \\
61404  & 237972  & 57120  & 214368  & 14616  & 1176  & -108024  & 54348  & 59640  & -25200  & -34440  & -20244  & 57708   \\
-90384  & 20412  & 159264  & 19152  & -49644  & 5124  & -11256  & -49616  & -31416  & -13944  & -5628  & 5208  & -39564   \\
55524  & -95580  & 10752  & -89184  & 40320  & 59388  & 259056  & 72744  & -75096  & -51002  & 61320  & -30072  & 55944   \\
-50064  & -47964  & -93048  & 2940  & 87108  & -18144  & -86520  & -78876  & -39396  & -39228  & -19152  & -9912  & 31164   \\
-145908  & -27780  & -183540  & -48480  & 225456  & -99624  & -3528  & -144998  & -44520  & 71232  & 8988  & 58884  & -8008   \\
41016  & -438732  & 230832  & 204708  & 135744  & 31332  & -19334  & 29316  & 158508  & 130298  & 18060  & -50883  & 141876   \\
-15456  & 514500  & -49980  & -85344  & 8568  & -178248  & 142128  & 23856  & 308196  & 298956  & -19467  & 151956  & -126483   \\
610584  & 21840  & 127932  & -65184  & -323834  & 195636  & 90972  & 339794  & -100716  & -96012  & 24864  & -114219  & 36876   \\
21840  & 466704  & -110628  & -106820  & -54012  & -90636  & -111790  & -14952  & 63672  & 107856  & -67788  & 61404  & -88284   \\
127932  & -110628  & 1045968  & 142632  & -410592  & 171024  & 86268  & 176820  & 96516  & 199752  & 13524  & -70784  & -42756   \\
-65184  & -106820  & 142632  & 569856  & 21518  & -30156  & -159264  & -23016  & 410004  & -71484  & -62076  & -13860  & 74032   \\
-323834  & -54012  & -410592  & 21518  & 604128  & -229992  & -75516  & -297276  & 182385  & 75684  & -3528  & 94500  & 138432   \\
195636  & -90636  & 171024  & -30156  & -229992  & 169512  & 104748  & 187341  & -136332  & -145719  & 35364  & -94836  & 24612   \\
90972  & -111790  & 86268  & -159264  & -75516  & 104748  & 381920  & 147168  & -182595  & -36876  & 105504  & -24612  & -7560   \\
339794  & -14952  & 176820  & -23016  & -297276  & 187341  & 147168  & 346248  & -59304  & -137928  & 64932  & -90888  & 28392   \\
-100716  & 63672  & 96516  & 410004  & 182385  & -136332  & -182595  & -59304  & 776776  & 48972  & -98784  & 19152  & 180852   \\
-96012  & 107856  & 199752  & -71484  & 75684  & -145719  & -36876  & -137928  & 48972  & 494536  & -28392  & 118188  & -130200   \\
24864  & -67788  & 13524  & -62076  & -3528  & 35364  & 105504  & 64932  & -98784  & -28392  & 60984  & 0  & -3780   \\
-114219  & 61404  & -70784  & -13860  & 94500  & -94836  & -24612  & -90888  & 19152  & 118188  & 0  & 74760  & -65100   \\
36876  & -88284  & -42756  & 74032  & 138432  & 24612  & -7560  &
28392  & 180852  & -130200  & -3780  & -65100  & 194040
\end{bmatrix}.
\end{equation}
\normalsize
\end{landscape}

\newpage

Next, we prove that the polynomial
$\tilde{y}^TH_{\tilde{f}}(\tilde{x})\tilde{y}$ in
(\ref{eq:y.H_f_tilda.y.xi^2.2.6.example}) is not sos. Let us first
present this polynomial and give it a name:

\begin{equation}\nonumber
\begin{array}{lll}
t(\tilde{x},\tilde{y})\mathrel{\mathop:}=\tilde{y}^TH_{\tilde{f}}(\tilde{x})\tilde{y}&=&294x_1x_2y_2^2-6995x_2^4y_1y_2-10200x_1y_1y_2-4356x_1^2x_2y_1^2-2904x_1^3y_1y_2
\\ \\ \ &\ &-11475x_1x_2^2y_1^2+13680x_2^3y_1y_2+4764x_1x_2y_1^2+4764x_1^2y_1y_2+6429x_1^2x_2^2y_1^2\\\\ \ &\ &+294x_2^2y_1y_2
-13990x_1x_2^3y_2^2-12123x_1^2x_2y_2^2-3872x_2y_1y_2+\frac{2143}{2}x_1^4y_2^2\\
\\ \ &\ &+20520x_1x_2^2y_2^2+29076x_1x_2y_1y_2
-24246x_1x_2^2y_1y_2+14901x_1x_2^3y_1y_2\\ \\ \ &\
&+15039x_1^2x_2^2y_1y_2+8572x_1^3x_2y_1y_2
+\frac{44703}{4}x_1^2x_2^2y_2^2+5013x_1^3x_2y_2^2\\ \\ \ &\
&+632y_1y_2-12360x_2y_2^2
-5100x_2y_1^2+\frac{147513}{4}x_2^2y_2^2+7269x_2^2y_1^2-45025x_2^3y_2^2\\
\\ \ &\ &+\frac{772965}{32}x_2^4y_2^2
+\frac{14901}{8}x_2^4y_1^2-1936x_1y_2^2-84x_1y_1^2+\frac{3817}{2}y_2^2+1442y_1^2\\
\\ \ &\ &+7269x_1^2y_2^2+4356x_1^2y_1^2-3825x_1^3y_2^2
-180x_1^3y_1^2+2310x_1^4y_1^2\\ \\ \ &\
&+5013x_1x_2^3y_1^2-22950x_1^2x_2y_1y_2-1505x_1^4y_1y_2-4041x_2^3y_1^2-3010x_1^3x_2y_1^2.
\end{array}
\end{equation}
Note that $t$ is a polynomial in $4$ variables of degree $6$ that
is quadratic in $\tilde{y}$. Let us denote the cone of sos
polynomials in $4$ variables $(\tilde{x},\tilde{y})$ that have
degree $6$ and are quadratic in $\tilde{y}$ by
$\hat{\Sigma}_{4,6}$, and its dual cone by $\hat{\Sigma}_{4,6}^*$.
Our proof will simply proceed by presenting a dual functional
$\xi\in\hat{\Sigma}_{4,6}^*$ that takes a negative value on the
polynomial $t$. We fix the following ordering of monomials in what
follows:
\begin{equation}\label{eq:monomial.ordering.t}
\begin{array}{ll}
v=&[          y_2^2,
             y_1y_2,
              y_1^2,
           x_2y_2^2,
          x_2y_1y_2,
           x_2y_1^2,
         x_2^2y_2^2,
        x_2^2y_1y_2,
         x_2^2y_1^2,
         x_2^3y_2^2,
        x_2^3y_1y_2,
         x_2^3y_1^2,
         x_2^4y_2^2,
        x_2^4y_1y_2,
         x_2^4y_1^2, \\
         \ &           x_1y_2^2,
          x_1y_1y_2,
           x_1y_1^2,
        x_1x_2y_2^2,
       x_1x_2y_1y_2,
        x_1x_2y_1^2,
      x_1x_2^2y_2^2,
     x_1x_2^2y_1y_2,
      x_1x_2^2y_1^2,
      x_1x_2^3y_2^2,
     x_1x_2^3y_1y_2,\\
         \ &      x_1x_2^3y_1^2,
         x_1^2y_2^2,
        x_1^2y_1y_2,
         x_1^2y_1^2,
           x_1^2x_2y_2^2,
     x_1^2x_2y_1y_2,
      x_1^2x_2y_1^2,
    x_1^2x_2^2y_2^2,
   x_1^2x_2^2y_1y_2,
    x_1^2x_2^2y_1^2,
         x_1^3y_2^2,\\
         \ &             x_1^3y_1y_2,
           x_1^3y_1^2,
      x_1^3x_2y_2^2,
     x_1^3x_2y_1y_2,
      x_1^3x_2y_1^2,
         x_1^4y_2^2,
        x_1^4y_1y_2,
         x_1^4y_1^2]^T.
\end{array}
\end{equation}
Let $\vec{t}$ represent the vector of coefficients of $t$ ordered
according to the list of monomials above; i.e., $t=\vec{t}^Tv$.
Using the same ordering, we can represent our dual functional
$\xi$ with the vector
\begin{equation}\nonumber
\begin{array}{ll}
c=&[   19338,
       -2485,
       17155,
        6219,
       -4461,
       11202,
        4290,
       -5745,
       13748,
        3304,
       -5404,
       13227,
        3594,\\
         \ &       -4776,
       19284,
               2060,
        3506,
        5116,
         366,
       -2698,
        6231,
        -487,
       -2324,
        4607,
         369,
       -3657,
        3534,
        6122,\\
         \ &
         659,
        7057,
                 1646,
        1238,
        1752,
        2797,
        -940,
        4608,
        -200,
        1577,
       -2030,
        -513,
       -3747,
             2541,
               15261,\\
         \ &                220,
        7834]^T.
\end{array}
\end{equation}
We have
\begin{equation}\nonumber
\langle\xi,t\rangle=c^{T}\vec{t}=-\frac{364547}{16}<0.
\end{equation}
On the other hand, we claim that $\xi\in\hat{\Sigma}_{4,6}^*$;
i.e., for any form $w\in\hat{\Sigma}_{4,6}$, we should have
\begin{equation}\label{eq:c.w>=0.t}
\langle\xi,w\rangle=c^{T}\vec{w}\geq0,
\end{equation}
where $\vec{w}$ here denotes the coefficients of $w$ listed
according to the ordering in (\ref{eq:monomial.ordering.t}).
Indeed, if $w$ is sos, then it can be written in the form
\[
w(x)=\tilde{z}^{T}\tilde{Q}\tilde{z}= \mathrm{Tr} \ \tilde{Q}
\cdot \tilde{z}\tilde{z}^{T},
\]
for some symmetric positive semidefinite matrix $\tilde{Q}$, and a
vector of monomials
\[
\tilde{z}=[  y_2,
     y_1,
     x_2y_2,
     x_2y_1,
     x_1y_2,
     x_1y_1,
     x_2^2y_2,
     x_2^2y_1,
     x_1x_2y_2,
     x_1x_2y_1,
     x_1^2y_2,
     x_1^2y_1   ]^T.
\]
It is not difficult to see that
\begin{equation}\label{eq:c.vec(w)=traceQzzz'.t}
c^{T}\vec{w}= \mathrm{Tr}  \, \tilde{Q} \cdot
(\tilde{z}\tilde{z}^{T}) \vert_c,
\end{equation}
where by $(\tilde{z}\tilde{z}^{T})\vert_c$ we mean a matrix where
each monomial in $\tilde{z}\tilde{z}^{T}$ is replaced with the
corresponding element of the vector $c$. This yields the matrix
\setcounter{MaxMatrixCols}{30} \scalefont{.8}
\[(\tilde{z}\tilde{z}^{T})\vert_c=
\begin{bmatrix}[r]
19338 & -2485 & 6219 & -4461 & 2060 & 3506 & 4290 & -5745 & 366 & -2698 & 6122 & 659 \\
-2485 & 17155 & -4461 & 11202 & 3506 & 5116 & -5745 & 13748 & -2698 & 6231 & 659 & 7057 \\
6219 & -4461 & 4290 & -5745 & 366 & -2698 & 3304 & -5404 & -487 & -2324 & 1646 & 1238 \\
-4461 & 11202 & -5745 & 13748 & -2698 & 6231 & -5404 & 13227 & -2324 & 4607 & 1238 & 1752 \\
2060 & 3506 & 366 & -2698 & 6122 & 659 & -487 & -2324 & 1646 & 1238 & -200 & 1577 \\
3506 & 5116 & -2698 & 6231 & 659 & 7057 & -2324 & 4607 & 1238 & 1752 & 1577 & -2030 \\
4290 & -5745 & 3304 & -5404 & -487 & -2324 & 3594 & -4776 & 369 & -3657 & 2797 & -940 \\
-5745 & 13748 & -5404 & 13227 & -2324 & 4607 & -4776 & 19284 & -3657 & 3534 & -940 & 4608 \\
366 & -2698 & -487 & -2324 & 1646 & 1238 & 369 & -3657 & 2797 & -940 & -513 & -3747 \\
-2698 & 6231 & -2324 & 4607 & 1238 & 1752 & -3657 & 3534 & -940 & 4608 & -3747 & 2541 \\
6122 & 659 & 1646 & 1238 & -200 & 1577 & 2797 & -940 & -513 & -3747 & 15261 & 220 \\
659 & 7057 & 1238 & 1752 & 1577 & -2030 & -940 & 4608 & -3747 &
2541 & 220 & 7834
\end{bmatrix},
\] \normalsize
which is positive definite. Therefore, equation
(\ref{eq:c.vec(w)=traceQzzz'.t}) along with the fact that
$\tilde{Q}$ is positive semidefinite implies that
(\ref{eq:c.w>=0.t}) holds. This completes the proof.

\end{document}